\RequirePackage{fix-cm}
\documentclass[a4paper]{article}       
\usepackage{graphicx}
\usepackage{amssymb}
\usepackage{color}
\usepackage{amsfonts}
\usepackage{amsmath}
\usepackage{framed}
\usepackage{bbm}
\usepackage{setspace}
\usepackage{theorem}
\newtheorem{theorem}{Theorem}[section]
\newtheorem{lemma}[theorem]{Lemma}

\newtheorem{proposition}[theorem]{Proposition}
\newtheorem{remark}[theorem]{Remark}
\newenvironment{proof}{{\bf
Proof:\,}}{\hspace*{\fill}\rule{1.2ex}{1.2ex}\\ }

\newtheorem{assumption}[theorem]{Assumption}

\newcommand{\di}{\text{div}}
\newcommand{\ve}{\textbf{v}}

\newcommand{\ue}{\textbf{u}}
\newcommand{\we}{\textbf{w}}

\newcommand{\weight}[1]{\langle #1\rangle}
\newcommand{\into}{\int \limits_{\Omega}}

\newcommand{\dx}{\mathit{dx}}
\newcommand{\dt}{\mathit{dt}}

\newcommand{\R}{\mathbb{R}}

\newcommand{\eps}{\varepsilon}
\begin{document}

\title{Local Well-Posedness of a Quasi-Incompressible Two-Phase Flow}



\author{Helmut Abels\thanks{  \textit{Fakult\"at f\"ur Mathematik,   Universit\"at Regensburg,   93040 Regensburg,   Germany}   \textsf {helmut.abels@ur.de} } \  and Josef Weber\thanks{\textit   {Fakult\"at f\"ur Mathematik,   Universit\"at Regensburg,   93040 Regensburg,   Germany}  } }

\date{
\emph{Dedicated to Matthias Hieber on the occasion of his 60th birthday}}

\maketitle

\begin{abstract}
  We show well-posedness of a diffuse interface model for a two-phase flow of two viscous incompressible fluids with different densities locally in time. The model leads to an inhomogeneous Navier-Stokes/Cahn-Hilliard system with a solenoidal velocity field for the mixture, but a variable density of the fluid mixture in the Navier-Stokes type equation. We prove existence of strong solutions locally in time with the aid of a suitable linearization and a contraction mapping argument. To this end we show maximal $L^2$-regularity for the Stokes part of the linearized system and use maximal $L^p$-regularity for the linearized Cahn-Hilliard system.
\end{abstract}


\section{Introduction and Main Result}

In this contribution we study a thermodynamically consistent, diffuse interface model for two-phase flows of two viscous incompressible system with different densities in a bounded domain in two or three space dimensions. The model was derived by A., Garcke and Gr\"un in \cite{MR2890451} and leads to the following inhomogeneous Navier-Stokes/Cahn-Hilliard system: 
\begin{alignat}{1}
\partial_t  (\rho  \ve) + & \di ( \rho \ve \otimes \ve)   + \di \Big ( \ve \otimes \tfrac{\tilde \rho_1 - \tilde \rho_2}{2} m(\varphi) \nabla (\tfrac{1}{\varepsilon} W'(\varphi) - \varepsilon \Delta \varphi ) \Big )     \nonumber \\ 
&= \di (- \varepsilon \nabla \varphi \otimes \nabla \varphi) +  \di (2 \eta (\varphi) D\ve) - \nabla q, \label{equation_strong_solutions_1} \\
\di \ve &= 0,  \\
\partial_t  \varphi +\ve\cdot \nabla \varphi &= \di  ( m( \varphi) \nabla \mu ),   \\
\mu &= - \varepsilon \Delta \varphi  + \frac{1}{\varepsilon} W' (\varphi  ) \label{equation_strong_solutions_2} 
\end{alignat}
in $Q_T:= \Omega\times (0,T)$ together with the initial and boundary values
\begin{align}
\ve_{|\partial \Omega} = \partial_n \varphi_{|\partial \Omega}= \partial_n \mu_{|\partial \Omega}  &= 0  && \text{ on } (0,T) \times \partial \Omega , \\
\varphi (0) = \varphi_0 , \ve(0) &= \ve_0    && \text{ in }  \Omega .  \label{equation_strong_solutions_inital_data_v}
\end{align}
Here  $\Omega \subseteq \mathbb R^d$, $d= 2,3$, is a bounded domain with  $C^4$-boundary. In this model the fluids are assumed to be partly miscible and $\varphi\colon \Omega\times (0,T)\to \R$ denotes the volume fraction difference of the fluids.
$\ve$, $q$, and $\rho$ denote the mean velocity, the pressure and the density of the fluid mixture. It is assumed that the density is a given function of $\varphi$, more precisely
\begin{equation*}
  \rho=\rho(\varphi) = \frac{\tilde{\rho}_1+\tilde{\rho}_2}2 +\frac{\tilde{\rho}_2-\tilde{\rho}_1}2 \varphi \qquad \text{for all }\varphi \in\R. 
\end{equation*}
where $\tilde{\rho}_1, \tilde{\rho}_2$ are the specific densities of the (non-mixed) fluids. Moreover, $\mu$ is a chemical potential and $W(\varphi)$ is a homogeneous free energy density  associated to the fluid mixture, $\eps>0$ is a constant related to ``thickness'' of the diffuse interface, which is described by $\{x\in \Omega: |\varphi(x,t)|<1-\delta\}$ for some (small) $\delta>0$, and $m(\varphi)$ is a mobility coefficient, which controls the strength of the  diffusion in the system. Finally $\eta(\varphi)$ is a viscosity coefficient and $D\ve= \frac12(\nabla \ve + \nabla \ve^T)$. 

Existence of weak solution for this system globally in time was shown by A., Depner, and Garcke in \cite{AbelsDepnerGarcke} and \cite{AbelsDepnerGarckeDegMob} for non-degenerate and degenerate mobility in the case of a singular free energy density $W$. Moreover, Gr\"un showed in   \cite{GruenAGG}  convergence (of suitable subsequences) of a fully discrete finite-element scheme for this system to a weak solution in the case of a smooth $W\colon \R \to \R$ with suitable polynomial growth. In the case of dynamic boundary conditions, which model moving contact lines, existence of weak solutions for this system was shown by Gal, Grasselli, and Wu in \cite{MR3981392}. In the case of non-Newtonian fluids of suitable $p$-growth existence of weak solutions was proved by A.\ and Breit \cite{AbelsBreit}. For the case of a non-local Cahn-Hilliard equation and Newtonian fluids the corresponding results was derived by Frigeri in \cite{MR3540647} and for a model with surfactants by Garcke and the authors in \cite{MR3845562}. Recently, Giorgini~\cite{GiorginiPreprint} proved existence of local strong solutions in a two-dimensional bounded, sufficiently smooth domain and global existence of strong solutions in the case of a two-dimensional torus.  

\begin{remark}
In \cite{AbelsDepnerGarcke} it is shown that the first equation is equivalent to
\begin{align}\nonumber
\rho \partial_t \ve  + \Big(  \rho \ve + \tfrac{\tilde \rho_1   - \tilde \rho_2}{2} m(\varphi) & \nabla   (\tfrac{1}{\varepsilon} W' (\varphi) - \varepsilon \Delta \varphi )  \Big) \cdot \nabla \ve +  \nabla p - \di (2 \eta (\varphi) D\ve) \\\label{eq:equivalent1}
&= - \varepsilon \Delta \varphi \nabla \varphi .
\end{align}
This reformulation will be useful in our analysis.  
\end{remark}


For the following we assume: 
\begin{assumption}\label{strong_solutions_general_assumptions} 
\begin{enumerate}
\item Let $\Omega \subseteq \mathbb R^d$ be a bounded domain with $C^4$-boundary and $d = 2,3$.
\item Let $\eta,m  \in C^5_b (\mathbb R)$ be such that $\eta (s) \geq \eta_0 > 0$ and $m(s)\geq m_0$ for every $s \in \mathbb R$ and some $\eta_0,m_0 > 0$.
\item The density $\rho\colon \R\to \R$ is given by
\begin{align*}
\rho = \rho (\varphi ) = \frac{\tilde \rho_1 + \tilde \rho_2}{2} + \frac{\tilde \rho_2 - \tilde \rho_1}{2} \varphi  \qquad \text{ for all } \varphi \in \mathbb R .
\end{align*}
\item $W\colon \R\to \R$ is five times continuously differentiable.
\end{enumerate}
\end{assumption}

With these assumptions we will show our main existence result on short time existence of strong solutions for \eqref{equation_strong_solutions_1}-\eqref{equation_strong_solutions_inital_data_v}:
\begin{theorem}[Existence of strong solutions]~\label{strong_solution_existence_proof_of_strong_solution}
\\
Let $\Omega$, $\eta$, $m$, $\rho$ and $W$ be as in Assumption \ref{strong_solutions_general_assumptions}. Moreover, let $\bold v_0 \in H^{1}_0 (\Omega)^d \cap L^2_\sigma (\Omega) $ and $\varphi_0 \in (L^p (\Omega) , W^4_{p,N} (\Omega))_{1- \frac{1}{p}, p}$ be given with $|\varphi_0(x)|\leq 1$ for all $x\in\Omega$ and $4 < p < 6$. Then there exists $T > 0$ such that  \eqref{equation_strong_solutions_1}-\eqref{equation_strong_solutions_inital_data_v} has a unique strong solution 
\begin{align*}
\bold v & \in W^1_2 (0,T; L^2_\sigma (\Omega)) \cap L^2 (0,T; H^2 (\Omega)^d \cap H^1_0 (\Omega)^d) , \\
\varphi & \in W^1_p (0,T; L^p (\Omega)) \cap L^p (0,T; W^4_{p,N} (\Omega)),
\end{align*}
where $W^4_{p,N}(\Omega)=\{u\in W^4_p(\Omega): \partial_n u|_{\partial\Omega}= \partial_n \Delta u|_{\partial\Omega}=0\}$.
\end{theorem}
We will prove this result with the aid of a contraction mapping argument after a suitable reformulation, similar to \cite{MR2504845}. But for the present system the linearized system is rather different. 

The structure of this contribution is as follows: In Section~\ref{sec:prelim} we introduce some basic notation and recall some results used in the following. The main result is proved in Section~\ref{sec:Main}. For its proof we use suitable estimates of the non-linear terms, which are shown in Section~\ref{sec:Lipschitz}, and a result on maximal $L^2$-regularity of a Stokes-type system, which is shown in Section~\ref{sec:Linear}.

\section{Preliminaries}\label{sec:prelim}

For an open set $U\subseteq \R^d$, $m\in\mathbb{N}_0$ and $1\leq p \leq\infty$ we denote by $W^m_p(U)$ the $L^p$-Sobolev space of order $m$ and $W^m_p(U;X)$ its $X$-valued variant, where $X$ is a Banach space. In particular, $L^p(U)=W^0_p(U)$ and $L^p(U;X)= W^0_p(U;X)$. Moreover, $B^s_{pq}(\Omega)$ denotes the standard Besov space, where $s\in\R$, $1\leq p,q\leq \infty$, and $L^2_\sigma(\Omega)$ is  the closure of $C^\infty_{0,\sigma}(\Omega)= \{\ue\in C^\infty_0(\Omega)^d: \operatorname{div} \ue=0\}$ in $L^2(\Omega)^d$ and $\mathbb{P}_\sigma\colon L^2(\Omega)^d\to L^2_\sigma(\Omega)$ the orthogonal projection onto it, i.e., the Helmholtz projection.

We will frequently use:
\begin{theorem}[Composition with Sobolev functions]~\label{theorem_composition_sobolev_functions}\\
Let $\Omega\subseteq \R^d$ be a bounded domain with $C^1$-boundary, $m,n\in \mathbb{N}$ and let $1\leq p <\infty$ such that $m - dp > 0$. Then for every $f\in C^m(\R^N)$ and every $R>0$ there exists a constant $C>0$ such that for all $u\in W^m_p(\Omega)^N$ with $\|u\|_{W^m_p(\Omega)^N}\leq R$, we have $f(u)\in W^m_p(\Omega)$ and $\|f(u)\|_{W^m_p(\Omega)}\leq C$.
Moreover, if $f\in C^{m+1}(\R^N)$, then for all $R>0$ there exists a constant $L>0$ such
that
\begin{equation*}
  \|f(u)-f(v)\|_{W^m_p(\Omega)}\leq L \|u-v\|_{W^m_p(\Omega)^N}
\end{equation*}
for all $u, v\in W^m_p(\Omega)^N$ with  $\|u\|_{W^m_p(\Omega)^N}, \|v\|_{W^m_p(\Omega)^N}\leq R$.
\end{theorem}
\begin{proof}
  The first part follows from \cite[Chapter 5, Theorem 1 and Lemma]{RunstSickel}. The second part can be easily reduced to the first part.
\end{proof}
In particular we have $uv\in W^m_p(\Omega)$ for all $u,v\in W^m_p(\Omega)$ under the assumptions of the theorem.

Let $X_0,X_1$ be Banach spaces such that $X_1\hookrightarrow X_0$ densely. It is well known that
\begin{equation} 
  \label{eq:BUCEmbedding}
   W^1_p(I;X_0) \cap L^p(I;X_1) \hookrightarrow BUC(I;(X_0,X_1)_{1-\frac1p,p}), \qquad 1\leq p <\infty,
\end{equation}
continuously for $I=[0,T]$, $0<T<\infty$,  and $I=[0,\infty)$, cf. Amann~\cite[Chapter III, Theorem 4.10.2]{Amann}. Here $(X_0,X_1)_{\theta,p}$ denotes the real interpolation space of $(X_0,X_1)$ with exponent $\theta$ and summation index $p$. Moreover, $BUC(I;X)$ is the space of all bounded and uniformly continuous $f\colon I\to X$ equipped with the supremum norm, where $X$ is a Banach space. 

Moreover, we will use:
\begin{lemma}\label{lemma_hoelder_sobolev_interpolation_to_hoelder}
Let $X_0 \subseteq Y \subseteq X_1$ be Banach spaces such that 
\begin{align*}
\|x\|_Y \leq C \|x\|^{1 - \theta}_{X_0} \|x\|^\theta_{X_1}
\end{align*}
for every $x \in X_0$ and a constant $C > 0$, where $\theta \in (0,1)$. Then 
\begin{align*}
C^{0, \alpha} ([0,T]; X_1) \cap L^\infty (0,T; X_0) \hookrightarrow C^{0, \alpha \theta} ([0,T] ; Y) . 
\end{align*}
continuously.
\end{lemma}
The result is well-known and can be proved in a straight forward manner.

\section{Proof of the Main Result}~\label{sec:Main}
We prove the existence of a unique strong solution $(\ve, \varphi) \in X_T$ for small $T > 0$, where the space $X_T$ will be specified later.
The idea for the proof is to linearize the highest order terms in the equations above at the initial data and then to split the equations in a linear and a nonlinear part such that
\begin{align*}
\mathcal L (\ve, \varphi ) = \mathcal F (\ve , \varphi ) ,
\end{align*}
where we still have to specify in which sense this equation has to hold.
To linearize it formally at the initial data we replace $\ve$, $p$ and $\varphi$ by $\ve_0 + \varepsilon \ve$, $p_0 + \varepsilon p$ and $\varphi_0 + \varepsilon \varphi$ and then differentiate with respect to $\varepsilon$ at $\varepsilon = 0$. 
In (\ref{equation_strong_solutions_1}) and the equivalent equation \eqref{eq:equivalent1}, the highest order terms with respect to $t$ and $x$ are
$\rho \partial_t \ve $, $ \di (2 \eta (\varphi) D\ve)$ and $\nabla p$. Hence the linearizations are given by
\begin{align*}
\frac{ d}{ \mathit{d \varepsilon}} \left ( \rho (\varphi_0 + \varepsilon \varphi ) \partial_t  (\ve_0 + \varepsilon \ve) \right )_{| \varepsilon = 0} &= \rho ' (\varphi_0) \varphi \partial_t \ve_0 + \rho (\varphi_0) \partial_t \ve =  \rho_0 \partial_t \ve , \\
\frac{d}{\mathit{d \varepsilon}} \left ( \di (2 \eta (\varphi_0 + \varepsilon \varphi) D(\ve_0 + \varepsilon \ve ) ) \right )_{| \varepsilon = 0} &= \di (2 \eta ' (\varphi_0) \varphi D\ve_0) +  \di (2 \eta (\varphi_0) D\ve ) , \\
\frac{d}{\mathit{d \varepsilon}} \nabla (p_0 + \varepsilon p)_{| \varepsilon = 0} &= \nabla p ,
\end{align*}
where $\rho_0 := \rho (\varphi_0)$ and $\rho_0 ' := \rho ' (\varphi_0)$. Moreover, we omit the term $\di (2 \eta ' (\varphi_0) \varphi D\ve_0) $ in the second linearization since it is of lower order. For the last equation we get the linearization
\begin{align*}
  &\frac{d}{d \tilde \varepsilon} \di (m (\varphi_0 + \tilde \varepsilon \varphi) \nabla ( - \varepsilon \Delta (\varphi_0 + \tilde \varepsilon \varphi )))_{|\tilde \varepsilon = 0}  &\\
  &\quad = -  \varepsilon  \di ( m ' (\varphi_0) \varphi \nabla \Delta \varphi_0 ) - \varepsilon \di ( m (\varphi_0) \nabla  \Delta \varphi) .
\end{align*}
We can omit the first term since it is of lower order. The second term can formally be reformulated as
\begin{align*}
- \varepsilon \di (m (\varphi_0) \nabla \Delta \varphi) = - \varepsilon m' (\varphi_0) \nabla \varphi_0 \cdot \nabla \Delta \varphi - \varepsilon m (\varphi_0) \Delta ( \Delta \varphi ).
\end{align*}
Here the first summand is of lower order again. Hence, the linearization is given by $- \varepsilon m (\varphi_0) \Delta ^2 \varphi$ upto terms of lower order.
Due to these linearizations we define the linear operator $\mathcal L \colon X_T \rightarrow Y_T$ by
\begin{align*}
\mathcal L (\ve, \varphi) = 
\begin{pmatrix}
\mathbb P_\sigma ( \rho_0 \partial_t \ve ) - \mathbb P_\sigma ( \di (2 \eta (\varphi_0) D\ve ))  \\
\partial_t \varphi + \varepsilon m(\varphi_0) \Delta^2  \varphi
\end{pmatrix} ,
\end{align*}
where $\mathcal L$ consists of the principal part of the lionization's, i.e., of the terms of the highest order. Furthermore, we define the nonlinear operator $\mathcal F \colon X_T \rightarrow Y_T$ by
\begin{align*}
\mathcal F ( \ve, \varphi) = 
\begin{pmatrix}
\mathbb P_\sigma F_1 (\ve, \varphi)  \\
- \nabla \varphi \cdot \ve + \di ( \tfrac{1}{\varepsilon} m(\varphi) \nabla W' (\varphi)) + \varepsilon  m (\varphi_0) \Delta^2 \varphi - \varepsilon \di ( m (\varphi) \nabla \Delta \varphi)
\end{pmatrix} ,
\end{align*}
where
\begin{align*}
F_1 (\ve, \varphi) = ( & \rho_0 - \rho ) \partial_t \ve  - \di (2 \eta (\varphi_0) D\ve ) + \di ( 2 \eta (\varphi) D\ve )  - \varepsilon \Delta \varphi \nabla \varphi \\
& - \left   (  \left ( \rho \ve + \tfrac{\tilde \rho_1   - \tilde \rho_2}{2} m(\varphi)  \nabla (\tfrac{1}{\varepsilon} W' (\varphi) - \varepsilon \Delta \varphi )  \right ) \cdot \nabla \  \right ) \ve .
\end{align*}
It still remains to define the spaces $X_T$ and $Y_T$. To this end, we set
\begin{align*}
Z^1_T &:= L^2 (0,T; H^2 (\Omega)^d \cap H^1_0 (\Omega)^d) \cap  W^1_2 (0,T; L^2_\sigma (\Omega)) , \\
Z^2_T &:= L^p (0,T; W^4_{p,N} (\Omega)) \cap W^1_p (0,T; L^p (\Omega))
\end{align*}
with $4 < p < 6$, where
\begin{align*}
W^4 _{p,N} (\Omega) := \{ \varphi \in W^4_p (\Omega) | \ \partial_n \varphi = \partial_n (\Delta \varphi ) = 0 \} .
\end{align*} 
We equip $Z^1_T$ and $Z^2_T$ with the norms $\|\cdot\|_{Z^1_T} '$ and $\|\cdot\|_{Z^2_T} '$ defined by
\begin{align}
\|\ve\|_{Z^1_T} ' & := \|\ve'\|_{L^2 (0,T; L^2 (\Omega))} + \|\ve\|_{L^2 (0,T; H^2 (\Omega))} + \|\ve(0)\|_{(L^2 (\Omega), H^2 (\Omega))_{\frac{1}{2}, 2}} , \nonumber \\
\|\varphi\|_{Z^2_T} ' & := \|\varphi '\|_{L^p (0,T; L^p (\Omega))} + \|\varphi\|_{L^p (0,T; W^4_{p,N} (\Omega))} + \|\varphi (0)\|_{(L^p (\Omega), W^4_p (\Omega))_{1 - \frac{1}{p}, p}} .\nonumber 
\end{align}
We use these norms since they guarantee that for all embeddings we will study later the embedding constant $C$ does not depend on $T$, cf. Lemma \ref{lemma_embedding_constant_does_not_depend_on_T}.
To this end we use:
\begin{lemma}\label{lemma_embedding_constant_does_not_depend_on_T}
Let $0 < T_0 < \infty$ be given and $X_0$, $X_1$ be some Banach spaces such that $X_1 \hookrightarrow X_0$ densely. For every $0 < T < \frac{ T_0}{2}$ we define 
\begin{align*} 
X_T := L^p (0,T; X_1) \cap W^1_p (0,T; X_0),
\end{align*}
where $1 \leq p < \infty$, equipped with the norm
\begin{equation*}
  \|u\|_{X_T}:= \|u\|_{L^p(0,T;X_1)}+\|u\|_{W^1_p(0,T;X_0)}+\|u(0)\|_{(X_0,X_1)_{1-\frac1p,p}}.
\end{equation*}
Then there exists an extension operator $E : X_T \rightarrow X_{T_0}$ and some constant $C > 0$ independent of $T$ such that $Eu_{|(0,T)} = u$ in $X_T$ and
\begin{align*}
\|Eu\|_{X_{T_0}} \leq C \|u\|_{X_T}
\end{align*}
for every $u \in X_T$ and every $0 < T < \frac{T_0}{2}$. Moreover, there exists a constant $\tilde C (T_0) > 0$ independent of $T$ such that
\begin{align*}
\|u\|_{BUC ([0,T]; (X_0, X_1)_{1 - \frac{1}{p},p})} \leq \tilde C (T_0) \|u\|_{X_T}
\end{align*}
for every $u \in X_T$ and every $0 <T < \frac{ T_0}{2}$. 
\end{lemma}
\begin{proof}
The result is well-known. In the case $u(0)=0$, one can prove the result with the aid of the extension operator defined by
\begin{align*}
(Eu) (t) := 
\begin{cases}
u(t) & \text{ if } t \in [0,T], \\
u(2T - t)  & \text{ if } t \in (T, 2T], \\
0 & \text{ if } t \in (2T, T_0] .
\end{cases}
\end{align*}
The case $u(0)\neq 0$ can be easily reduced to the case $u(0)=0$ by substracting a suitable extension of $u_0$ to $[0,\infty)$. We refer to \cite[Lemma~5.2]{Dissertation_Weber} for the details.
\end{proof}

The last preparation before we can start with the existence proof is the definition of the function spaces $X_T := X_T^1 \times X_T^2$ and $Y_T$ by
\begin{align*}
X_T^1 & := \{ \ve \in  Z^1_T |  \ \ve_{|t=0} = \ve_0 \} , \\
X_T^2 & := \{ \varphi \in Z^2_T | \ \varphi_{| t = 0} = \varphi_0 \} , \\
Y_T & := Y^1_T \times Y^2_T := L^2 (0,T; L^2_\sigma (\Omega)) \times L^p (0,T; L^p (\Omega)) ,
\end{align*}
where 
\begin{align*}
\ve_0 \in (L^2_\sigma (\Omega), H^2 (\Omega)^d \cap H^1_0 (\Omega)^d \cap L^2_\sigma (\Omega))_{\frac{1}{2}, 2} = H^1_0 (\Omega)^d \cap L^2_\sigma (\Omega)
\end{align*}
and 
\begin{align*}
\varphi_0 \in (L^p (\Omega), W^4_{p,N} (\Omega) )_{1 - \frac{1}{p}, p}
\end{align*}
are the initial values from (\ref{equation_strong_solutions_inital_data_v}).
Note that in the space $X^2_T$ we have to ensure that $\varphi_{|t=0} = \varphi_0 \in [-1,1]$ since we will use this property to show the Lipschitz continuity of $\mathcal F : X_T \rightarrow Y_T$ in Proposition \ref{strong_solution_proposition_lipschitz_continuity_F}.
Moreover, we note that $X_T$ is not a vector space due to the condition $\varphi_{|t=0} = \varphi_0$. It is only an affine linear subspace of $Z_T := Z^1_T \times Z^2_T$.

\begin{proposition}\label{strong_solution_proposition_lipschitz_continuity_F}
Let the Assumptions \ref{strong_solutions_general_assumptions} hold and $\varphi_0$ be given as in  Theorem \ref{strong_solution_existence_proof_of_strong_solution}. Then there is a constant $C (T,R) > 0$ such that
\begin{align}\label{strong_solution_estimate_lipschitz_continuity_F}
\| \mathcal F (\bold v_1 , \varphi_1 ) - \mathcal F (\bold v_2 , \varphi_2) \| _{Y_T} \leq C (T, R) \|(\bold v_1 -  \bold v_2, \varphi_1 - \varphi_2 )\|_{X_T}
\end{align}
for all $(\bold v_i, \varphi_i) \in X_T$ with $\|(\bold v_i, \varphi_i)\|_{X_T} \leq R$ and $i = 1,2$. Moreover, it holds $C(T,R) \rightarrow 0$ as $T \rightarrow 0$.
\end{proposition}
The proposition is proved in Section~\ref{sec:Lipschitz} below.

\begin{theorem}\label{thm:linear}
  Let $T>0$ and $\mathcal{L}$, $X_T$ and $Y_T$ be defined as before. Then $\mathcal{L}\colon X_T\to Y_T$ is invertible. Moreover, for every $T_0>0$ there is a constant $C(T_0)>0$ such that
  \begin{equation*}
    \|\mathcal{L}^{-1}\|_{\mathcal{L}(Y_T,X_T)}\leq C(T_0)\qquad \text{for all }T\in (0,T_0].
  \end{equation*}
\end{theorem}
This theorem is proved in Section~\ref{sec:Linear} below.

\noindent
\emph{Proof of Theorem \ref{strong_solution_existence_proof_of_strong_solution}:}
First of all we note that  \eqref{equation_strong_solutions_1}-\eqref{equation_strong_solutions_2} is equivalent to
\begin{align}
&(\ve, \varphi) = \mathcal L^{-1} (\mathcal F (\ve, \varphi))  && \text{in } X_T . \label{strong_solution_fix_point_equation2}
\end{align}
The fact that $\mathcal L$ is invertible will be proven later.
Equation (\ref{strong_solution_fix_point_equation2}) implies that we have rewritten the system to a fixed-point equation which we want to solve by using the Banach fixed-point theorem.

To this end, we consider some $(\tilde \ve, \tilde \varphi) \in X_T$ and define
\begin{align*}
M := \| \mathcal L^{-1 } \circ \mathcal F (\tilde \ve, \tilde \varphi)\|_{X_T} < \infty .
\end{align*}
Now let $R > 0$ be given such that $(\tilde \ve, \tilde \varphi) \in \overline{B_R^{X_T} (0)}$ and $R > 2M$.
Then it follows from Proposition \ref{strong_solution_proposition_lipschitz_continuity_F} that there exists a constant $C = C(T, R) > 0$ such that
\begin{align*}
\| \mathcal F (\ve_1 , \varphi_1 ) - \mathcal F (\ve_2 , \varphi_2) \| _{Y_T} \leq C (T, R) \|(\ve_1, \varphi_1)  - ( \ve_2, \varphi_2)\|_{X_T}
\end{align*}
for all $(\ve_i, \varphi_i) \in X_T$ with $\|(\ve_i, \varphi_i)\|_{X_T} \leq R$, $j = 1,2$, where it holds $C(T,R) \rightarrow 0$ as $T \rightarrow 0$.
Furthermore, we choose $T$ so small that 
\begin{align*}
\|\mathcal L^{-1}\|_{\mathcal L (Y_T, X_T)} C (T,R) < \frac{1}{2} .
\end{align*}
Here we have to ensure that $\|\mathcal L^{-1}\|_{\mathcal L (Y_T, X_T)} $ does not converge to $+ \infty$ as $T \rightarrow 0$. But since Lemma \ref{strong_solution_lemma_l_inverse_t_t_0_firstpart} and Lemma \ref{strong_solution_lemma_l_inverse_t_t_0_secondpart} below yield  $\|\mathcal L^{-1}\|_{\mathcal L (Y_T, X_T)}  < C (T_0)$ for every $0 < T < T_0$ and for a constant that does not depend on $T$, this is not the case and we can choose $T > 0$ in such a way that the previous estimate holds.
Note that $T$ depends on $R$ and in general $T$ has to become smaller the larger we choose $R$. 
\\
Since we want to apply the Banach fixed-point theorem on $\overline{B^{X_T}_R (0)} \subseteq X_T$ as we only consider functions $(\ve, \varphi) \in X_T$ which satisfy $\|(\ve, \varphi)\|_{X_T} \leq R$, we have to show that $\mathcal L^{-1} \circ \mathcal F$ maps from $\overline{B^{X_T}_R (0)}$ to $\overline{B^{X_T}_R (0)}$.

From the considerations above we know that there exists $(\tilde \ve, \tilde \varphi) \in \overline{B^{X_T}_R (0)}$ such that 
\begin{align}\label{strong_solution_tilde_v0_tilde_phi0_bounded}
\|\mathcal L^{-1} \circ \mathcal F (\tilde \ve, \tilde \varphi) \|_{X_T} = M < \frac{R}{2} .
\end{align}
Then a direct calculation shows
\begin{align*}
\|\mathcal L ^{-1} \circ \mathcal F (\ve, \varphi) \|_{X_T} & \leq \| \mathcal L ^{-1} \circ \mathcal F ( \ve, \varphi) - \mathcal L ^{-1} \circ \mathcal F (\tilde \ve, \tilde \varphi) \|_{X_T} + \|   \mathcal L^{-1} \circ \mathcal F (\tilde \ve, \tilde \varphi) \|_{X_T} \\
& < \|\mathcal L^{-1} \|_{\mathcal L (Y_T, X_T)} \| \mathcal F (\ve, \varphi) - \mathcal F (\tilde \ve , \tilde \varphi) \|_{Y_T} + \frac{R}{2} \\
& \leq  \|\mathcal L^{-1} \|_{\mathcal L (Y_T, X_T)} C (R,T) \|(\ve, \varphi) - (\tilde \ve, \tilde \varphi)\|_{X_T} + \frac{R}{2}  <  R 
\end{align*}
for every $(\ve, \varphi) \in \overline{B^{X_T}_R (0)}$, where we used the estimate for the Lipschitz continuity of $\mathcal F$.
This shows that $\mathcal L ^{-1} \circ \mathcal F (\ve, \varphi)$ is in $ \overline{B^{X_T}_R (0)}$ for every $(\ve, \varphi) \in \overline{ B^{X_T} _R (0)}$, i.e.,
\begin{align*}
\mathcal L^{-1 } \circ \mathcal F :  \overline{ B^{X_T}_R (0) } \rightarrow  \overline{ B^{X_T} _R (0) } .
\end{align*}
For applying the Banach fixed-point theorem it remains to show that the mapping $\mathcal L^{-1 } \circ F \colon B^{X_T}_R (0) \rightarrow B^{X_T} _R (0) $ is a contraction. To this end, let $(\ve_i, \varphi_i) \in B^{X_T}_R (0)$ be given for $i = 1,2$. Then it holds
\begin{align*}
\| \mathcal L^{-1} & \circ \mathcal F (\ve_1, \varphi_1)  - \mathcal L^{-1} \circ \mathcal F  (\ve_2 , \varphi_2) \|_{X_T} \\
& \leq \| \mathcal L^{-1}\|_{\mathcal L (Y_T, X_T)} C(R,T) \|(\ve_1, \varphi_1) - (\ve_2, \varphi_2) \|_{X_T} \\
& < \frac{1}{2} \|(\ve_1, \varphi_1) - (\ve_2, \varphi_2) \|_{X_T}  ,
\end{align*}
which shows the statement. Hence, the Banach fixed-point theorem can be applied and yields some $(\ve, \varphi) \in \overline{B^{X_T}_R (0)} \subseteq X_T$ such that the fixed-point equation (\ref{strong_solution_fix_point_equation2}) holds, which implies that $(\ve, \varphi)$ is a strong solution for the equations \eqref{equation_strong_solutions_1}-\eqref{equation_strong_solutions_2}.

Finally, in order to show uniqueness in $X_T$, let $(\hat \ve, \hat \varphi)\in X_T$ be another solution. Choose $\hat{R}\geq R$ such that $(\hat \ve, \hat \varphi)\in \overline{B_{\hat{R}}^{X_T}(0)}$. Then by the previous arguments we can find some $\hat{T}\in (0,T]$ such that (\ref{strong_solution_fix_point_equation2}) has a unique solution. This implies $(\hat \ve, \hat \varphi)|_{[0,\hat{T}]}= (\ve, \varphi)|_{[0,\hat{T}]}$. A standard continuation argument shows that the solutions coincide for all $t\in [0,T]$.


\section{Lipschitz Continuity of $\mathcal F$}\label{sec:Lipschitz}

Before we continue we study in which Banach spaces $\ve$, $\varphi$, $\nabla \varphi$, $m (\varphi)$ and so on are bounded.

Note that in the definition of $X^2_T$, $p$ has to be larger than $4$ because we will need to estimate terms like $\nabla \Delta \varphi  \cdot \nabla \ve$, where $p = 2$ is not sufficient for the analysis and therefore we need to choose $p > 2$. 
But for most terms in the analysis $p=2$ would be sufficient and $4 < p < 6$ would not be necessary. Nevertheless, for consistency all calculations are done for the case $4 < p < 6$.

Due to \eqref{eq:BUCEmbedding} it holds
\begin{align}\label{strong_solution_v_buc_h1}
\ve \in X^1_T \hookrightarrow BUC ([0,T]; B^1_{22} (\Omega)) = BUC ([0,T]; H^1 (\Omega)) ,
\end{align}
where we used $B^s_{22} (\Omega) = H^s_2 (\Omega)$ for every $s \in \mathbb R$. In particular this implies
\begin{align}
&\nabla \ve \in L^\infty (0,T; L^2 (\Omega)) \cap L^2 (0,T; L^6 (\Omega)) \hookrightarrow L^{\frac{8}{3}} (0,T; L^4 (\Omega)) , \label{strong_solution_nabla_v_l83_l4}  \\
& \nabla \ve \in L^\infty (0,T; L^2 (\Omega)) \cap L^2 (0,T; L^6 (\Omega)) \hookrightarrow L^4 (0,T; L^3 (\Omega)) \label{strong_solution_nabla_v_l4_l3}.
\end{align}

Let $\varphi \in  X^2_T$ be given. From  it \eqref{eq:BUCEmbedding} follows
\begin{align}\label{strong_solution_varphi_buc_W4-4p_p}
\varphi  \in  L^p  (0,T; W^4_{p,N} (\Omega) ) \cap W^1_p (0,T; L^p (\Omega)) \hookrightarrow  BUC ([0,T]; W^{4 - \frac{4}{p}}_p (\Omega)) .
\end{align}
This implies
\begin{align}\label{strong_solution_nabla_delta_phi}
\nabla \Delta \varphi \in BUC ([0,T]; W^{1 - \frac{4}{p}}_p (\Omega)) 
\end{align}
since  $p > 4$.
Note that when we write ``$\varphi$ is bounded in $Z$" for some function space $Z$, we mean that the set of all functions $\{\varphi \in X_T^2 : \ \|\varphi\|_{X^2_T} \leq R\}$ is bounded in $Z$ in such a way that the upper bound only depends on $R$ and not on $T$, i.e., there exists $C(R) > 0$ such that $\|\varphi\|_{Z} \leq C(R)$ for every $\varphi \in X_T^2$ with $\|\varphi\|_{X^2_T} \leq R$.

First of all, we have
\begin{align*}
\varphi \in W^1_p (0,T; L^p (\Omega)) \hookrightarrow C^{0, 1 - \frac{1}{p}} ([0,T]; L^p (\Omega)) .
\end{align*}
Moreover, we already know that $\varphi \in BUC([0,T]; W^{4 - \frac{4}{p}}_p (\Omega))$ and we have
\begin{align*}
(B^{4 - \frac{4}{p}}_{pp}  (\Omega), L^p(\Omega))_{\theta, 2} = B^3_{p2} (\Omega) \hookrightarrow W^3_p (\Omega) 
\end{align*}
together with the estimate
\begin{align*}
\|\varphi (t)\|_{W^3_p (\Omega)} \leq C \|\varphi (t)\|^{1 - \theta}_{W^{4 - \frac{4}{p}}_p (\Omega)} \|\varphi (t)\|^\theta_{L^p (\Omega)} 
\end{align*}
for every $t \in [0,T]$. Hence, Lemma \ref{lemma_hoelder_sobolev_interpolation_to_hoelder} implies
\begin{align}\nonumber
  \varphi \in  &C^{0, 1 - \frac{1}{p}} ([0,T]; L^p (\Omega)) \cap  C([0,T]; W^{4 - \frac{4}{p}}_p (\Omega))\\\label{phi_hoelder_continuous_in_time_w3p}
  &\hookrightarrow C^{0, (1 - \frac{1}{p}) \theta} ([0,T]; W^3_p (\Omega)) .
\end{align}
Because of $W^3_p (\Omega) \hookrightarrow C^2 (\overline \Omega)$ for $d=2,3$ due to $4<p<6$, we obtain that
\begin{align}\label{phi_continuous_in_time_twice_in_omega}
\varphi &\text{ is bounded in } C([0,T]; C^2 (\overline \Omega))   .
\end{align}

In the nonlinear operator $\mathcal F \colon X_T \rightarrow Y_T$ the terms $\eta (\varphi)$, $\eta (\varphi_0)$, $m (\varphi)$, $m (\varphi_0)$ and $W' (\varphi)$ appear. 
Hence, we need to know in which spaces these terms are bounded in the sense that there is a constant $C(R) > 0$, which does not depend on $T$, such that the norms of these terms in a certain Banach space are bounded by $C(R)$ for every $(\ve, \varphi) \in X_T$ with $\|(\ve, \varphi)\|_{X_T} \leq R$.

Due to (\ref{phi_hoelder_continuous_in_time_w3p}) and because the embedding constant only depends on $R$, it holds 
\begin{align*}
\|\varphi (t)\|_{W^3_p (\Omega)} \leq C(R)
\end{align*}
for every $t \in [0,T]$ and $\varphi \in X_T^2$ with $\|\varphi\|_{X^2_T} \leq R$. Hence Theorem~\ref{theorem_composition_sobolev_functions} yields
\begin{align*}
\|f(\varphi (t))\|_{W^3_p (\Omega)}, \|f (\varphi_0)\|_{W^3_p (\Omega)}, \|W' (\varphi (t))\|_{W^3_p (\Omega)} \leq C(R)
\end{align*}
for every $t \in [0,T]$ and every $\varphi \in X^2_T$ with $\|\varphi\|_{X^2_T} \leq R$, where $ f \in \{\eta, m\}$. Thus
\begin{align}\label{f_of_phi_bounded_in_infty_w3p}
f (\varphi), f(\varphi_0), f'(\varphi), W' (\varphi)  \text{ are bounded in } L^\infty (0,T; W^3_p (\Omega))  
\end{align}
for $f \in \{\eta, m\}$.
Moreover, Theorem \ref{theorem_composition_sobolev_functions} yields the existence of $L > 0$ such that
\begin{align}\label{lipschitz_f_of_phi_w3p}
\|f (\varphi_1 (t)) - f (\varphi_2 (t))\|_{W^3_p (\Omega)} & \leq L \|\varphi_1 (t) - \varphi_2 (t)\|_{W^3_p (\Omega)} 
\end{align}
for every $t \in [0,T]$, $\varphi_1, \varphi_2 \in X^2_T$ and $f \in  \{ \eta, m , W' \}$.

In the next step, we want to show that $f(\varphi)$ is bounded in $X_T^2$ and therefore the same embeddings hold as for $\varphi$, where $f \in \{\eta, m , W'\}$. Note that from now on until the end of the proof of the interpolation result for $f(\varphi)$, we always use some general $f \in C^4_b (\mathbb R)$. But all these embeddings are valid for $f \in \{\eta, m, W'\}$.
We want to prove that if it holds $\varphi \in X^2_T$ with $\|\varphi\|_{X^2_T} \leq R$, then there exists a constant $C(R) > 0$ such that $\|f(\varphi)\|_{X^2_T} \leq C(R)$. To this end, let $\varphi \in X^2_T$ be given with $\|\varphi\|_{X^2_T} \leq R$. Since we already know $\varphi \in C([0,T]; C^2 (\overline \Omega))$, cf. (\ref{phi_continuous_in_time_twice_in_omega}), we can conclude 
\begin{align*}
\|\varphi (t)\|_{C^2 (\overline \Omega)} \leq C (R)
\end{align*}
for all $t \in [0,T]$. Hence, it holds $f (\varphi (t)) \in C^2 (\overline \Omega)$ for every $t \in [0,T]$ and
\begin{align*}
\nabla f ( (\varphi (t)) = f' (\varphi (t)) \nabla \varphi (t) .
\end{align*}
Due to (\ref{f_of_phi_bounded_in_infty_w3p}), $f ' (\varphi)$ is bounded in $L^\infty (0,T; W^3_p (\Omega))$. In particular, this implies $\|f' (\varphi (t))\|_{W^3_p (\Omega)} \leq C(R)$ for a.e. $t \in (0,T)$ and a constant $C(R) > 0$. Since it holds $\varphi \in L^p (0,T; W^4_p (\Omega))$, it follows $\nabla \varphi (t) \in W^3_p (\Omega)$ for a.e. $t \in (0,T)$. Since $W^3_p(\Omega)$ is a Banach algebra,   we obtain $f' (\varphi (t)) \nabla \varphi (t) \in W^3_p (\Omega)$ for a.e. $t \in (0,T)$ together with the estimate
\begin{align*}
\|\nabla f (\varphi (t))\|_{W^3_p (\Omega)} = \|f' (\varphi (t)) \nabla \varphi (t)\|_{W^3_p (\Omega)} \leq C \|f ' (\varphi (t))\|_{W^3_p (\Omega)} \|\nabla \varphi (t)\|_{W^3_p (\Omega)}
\end{align*}
for a.e. $t \in (0,T)$ and every $\varphi \in X^2_T$ with $\|\varphi\|_{X^2_T} \leq R$. Since $f' (\varphi)$ is bounded in $L^\infty (0,T; W^3_p (\Omega))$ and $\nabla \varphi$ is bounded in $L^p (0,T; W^3_p (\Omega))$, the estimate above implies the boundedness of $\nabla f (\varphi)$ in $L^p (0,T; W^3_p (\Omega))$, i.e., there exists $C(R) > 0$ such that
\begin{align*}
\|\nabla f (\varphi)\|_{L^p (0,T; W^3_p (\Omega))} \leq C (R) \qquad \text{ for all } \varphi \in X^2_T \text{ with } \|\varphi\|_{X^2_T} \leq R. 
\end{align*}
Altogether this implies that
\begin{align*}
f(\varphi) \text{ is bounded in } L^p (0,T; W^4_p (\Omega)) .
\end{align*}
Analogously we can conclude from the boundedness of $\varphi $ in $W^1_p (0,T; L^p (\Omega))$ that $f (\varphi) $ is also bounded in $W^1_p (0,T; L^p (\Omega))$ because of 
$
\frac{d}{dt} f (\varphi (t)) = f' (\varphi (t)) \partial_t \varphi (t), 
$
where $f' (\varphi ) $ is bounded in $C^0 (\overline Q_T) $. Thus the same interpolation result holds as in (\ref{phi_hoelder_continuous_in_time_w3p}), i.e., 
\begin{align}\label{interpolation_f_of_phi_hoelder_continuous_and_w3p}
f(\varphi) \text{ is bounded in } C^{0, ( 1 - \frac{1}{p} ) \theta} ([0,T]; W^3_p (\Omega)) ,
\end{align}
where $\theta :=  \frac{\frac{4}{p} - 1}{ \frac{4}{p} - 4}$.

\noindent
\emph{Proof of Proposition~\ref{strong_solution_proposition_lipschitz_continuity_F}:}
Let $(\ve_i, \varphi_i ) \in X_T$ with $\|(\ve_i, \varphi_i)\|_{X_T} \leq R$, $i = 1,2$, be given. Then it holds
\begin{align}\label{equation_F_lipschitz}
 &\| \mathcal F (\ve_1 , \varphi_1 ) -  \mathcal F (\ve_2 , \varphi_2) \| _{Y_T}  = \| \mathbb  P_\sigma ( F_1 (\ve_1, \varphi_1 ) - F_1 (\ve_2, \varphi_2)) \|_{L^2 (Q_T)} \nonumber \\
& \ \   + \|(\nabla \varphi_2 \cdot \ve_2 - \nabla \varphi_1 \cdot \ve_1 )  + \tfrac{1}{\varepsilon} \di ( m (\varphi_1) \nabla W' (\varphi_1) - m (\varphi_2) \nabla W' (\varphi_2))  \nonumber \\
& \ \  + \varepsilon m (\varphi_0) \Delta^2 (\varphi_1 - \varphi_2)  + \varepsilon \di  ( m (\varphi_2) \nabla \Delta \varphi_2 - m (\varphi_1) \nabla \Delta \varphi_2)\|_{L^p (Q_T)} .
\end{align}
For the sake of clarity we study both summands in (\ref{equation_F_lipschitz}) separately and begin with the first one. Recall that the operator $F_1$ is defined by
\begin{align*}
F_1 (\ve, \varphi) = & \rho_0 \partial_t \ve - \rho \partial_t \ve - \di (2 \eta (\varphi_0) D\ve ) + \di ( 2 \eta (\varphi) D\ve )  - \varepsilon \Delta \varphi \nabla \varphi  \\
& - \left   (  \left ( \rho \ve + \tfrac{\tilde \rho_1   - \tilde \rho_2}{2} m(\varphi)  \nabla (\tfrac{1}{\varepsilon} W' (\varphi) - \varepsilon \Delta \varphi )  \right ) \cdot \nabla \  \right ) \ve 
\end{align*}
and that it holds $\|\mathbb P_\sigma \|_{\mathcal L (L^2 (\Omega)^d, L^2_\sigma (\Omega))} \leq 1$  for the Helmholtz projection $\mathbb P_\sigma$.
We estimate $ \| \mathbb P_\sigma ( F_1 (\ve_1, \varphi_1 ) - F_1 (\ve_2, \varphi_2 ) )\|_{L^2 (Q_T)}$:

For the first two terms we can calculate
\begin{align*}
\|\rho_0 & \partial_t \ve_1 - \rho (\varphi_1) \partial_t \ve_1 - \rho_0 \partial_t \ve_2 + \rho (\varphi_2) \partial_t \ve_2 \|_{L^2(Q_T)} \\
& \leq  \| (\rho_0 - \rho (\varphi_1)) \partial_t (\ve_1 - \ve_2) \|_{L^2 (Q_T)} + \|  ( \rho (\varphi_1) - \rho (\varphi_2)) \partial_t \ve_2  \|_{L^2 (Q_T)} .
\end{align*}
Since it holds $\partial_t \ve_i \in L^2 (0,T; L^2_\sigma (\Omega))$, $i = 1,2$, we need to estimate every $\rho$-term in the $L^\infty$-norm. To this end, we use that $\rho$ is affine linear and 
\begin{align*}
\varphi_i \text{ is bounded in } C^{0, (1 -\frac{1}{p})\theta} ([0,T]; W^3_p (\Omega)) \hookrightarrow  C^{0, (1 -\frac{1}{p})\theta} ([0,T]; C^2 ( \overline \Omega))  
\end{align*}
for $i = 1,2$ and $\theta = \frac{\frac{4}{p}-1}{\frac{4}{p}-4}$, cf. (\ref{phi_hoelder_continuous_in_time_w3p}). Then we obtain for the first summand
\begin{align*}
 \| (\rho_0 - \rho (\varphi_1))  \partial_t (\ve_1 - \ve_2) \|_{L^2 (Q_T)} &  \leq \|\rho (\varphi_0) - \rho (\varphi_1)\|_{L^\infty (Q_T)} \|\partial_t (\ve_1 - \ve_2)\|_{L^2 (Q_T)} \\
& \leq C  \underset{t \in [0,T]}{\sup} \|\varphi_1 (0) - \varphi_1 (t)\|_{L^\infty (\Omega)} \|\ve_1 - \ve_2\|_{X^1_T} \\
& \leq C T^{ (1 - \frac{1}{p} ) \theta} \|\varphi_1\|_{C^{0, (1 - \frac{1}{p}) \theta} ([0,T]; C^2 ( \overline \Omega))}  \|\ve_1 - \ve_2\|_{X^1_T} \\
& \leq C R  T^{ (1 - \frac{1}{p} ) \theta}  \|\ve_1 - \ve_2\|_{X^1_T} .
\end{align*}
 Analogously the second term can be estimated by
\begin{align*}
 \| & ( \rho (\varphi_1) -  \rho (\varphi_2)) \partial_t \ve_2  \|_{L^2 (Q_T)}  \leq \|\rho (\varphi_1) - \rho (\varphi_2)\|_{L^\infty (Q_T)} \|\ve_2\|_{X^1_T} \\
& \leq C \underset{t \in [0,T]}{\sup} \| (\varphi_1(t) - \varphi_2 (t)) - (\varphi_1 (0) - \varphi_2 (0))\|_{L^\infty (\Omega)}   \|\ve_2\|_{X^1_T} \\
& \leq C R T^{(1 - \frac{1}{p}) \theta}   \|\varphi_1 - \varphi_2\|_{C^{0, (1 - \frac{1}{p}) \theta} ([0,T]; C^2 ( \overline \Omega))} \\
& \leq C R T^{ (1 - \frac{1}{p}) \theta} \|\varphi_1 - \varphi_2\|_{X^2_T} .
\end{align*}
Here we used the fact that $\varphi_1 (0) = \varphi_0 = \varphi_2 (0)$ for $\varphi_i \in X^2_T$, $i = 1,2$.

The next term of $ \| \mathbb P_\sigma ( F_1 (\ve_1, \varphi_1 ) - F_1 (\ve_2, \varphi_2 ) )\|_{L^2 (Q_T)}$ is given by
\begin{align*}
|&|  ( \di (2 \eta  (\varphi_0) D\ve_2)  - \di (2 \eta (\varphi_0) D\ve_1 ) ) +  ( \di (2 \eta  (\varphi_1) D\ve_1)  - \di (2 \eta (\varphi_2) D\ve_2 ) ) \|_{Y^1_T}  \\
& \leq \|  \di (2  (\eta (\varphi_0) -  \eta (\varphi_1) ) ( D\ve_2 - D\ve_1 ) ) \|_{Y^1_T} + \| \di (2 ((\eta (\varphi_1) - \eta (\varphi_2)) D\ve_2 ) ) \|_{Y^1_T}  .
\end{align*}
In the next step we apply the divergence on the $\eta(\varphi_i)$- and $D\ve_i$-terms and for the sake of clarity we study both terms in the previous inequality separately. For the first one we use $\eta (\varphi) \in C^{0, (1 -\frac{1}{p}) \theta} ([0,T]; W^3_p (\Omega))$ with $\theta = \frac{\frac{4}{p}-1}{\frac{4}{p} - 4}$, cf. (\ref{interpolation_f_of_phi_hoelder_continuous_and_w3p}), to obtain
\begin{align*}
 \| & \di (2  (\eta (\varphi_0)  -  \eta (\varphi_1) ) ( D\ve_2 - D\ve_1 ) ) \|_{Y^1_T} \\
& \leq \| 2  \nabla  (\eta (\varphi_0) - \eta (\varphi_1 ) ) \cdot (D\ve_2 - D\ve_1)  \|_{Y^1_T} + \|  (\eta (\varphi_0) - \eta (\varphi_1)) \Delta (\ve_2 - \ve_1 )  \|_{Y^1_T} \\
& \leq C \underset{t \in [0,T]}{\sup} \|  \nabla \eta (\varphi_1 (0)) - \nabla \eta (\varphi_1 (t)) \|_{C^1 ( \overline \Omega)} \| D\ve_2 - D\ve_1 \|_{L^2 (0,T;  H^1 (\Omega))} \\
& \ \ \ + C \underset{t \in (0,T)}{\sup} \|   \eta (\varphi_1 (0)) -  \eta (\varphi_1 (t)) \|_{C^2 ( \overline \Omega)} \| \Delta (\ve_2 - \ve_1) \|_{L^2 (0,T; L^2 (\Omega))} \\
& \leq C  T^{ (1 - \frac{1}{p}) \theta}  \|\nabla \eta (\varphi_1)\|_{C^{0, (1 - \frac{1}{p}) \theta } ([0,T]; W^2_p (\Omega))} \|\ve_1 - \ve_2\|_{X^1_T} \\
& \ \ \ + C  T^{ (1 - \frac{1}{p}) \theta}  \|\eta (\varphi_1)\|_{C^{0, (1 - \frac{1}{p}) \theta} ([0,T]; W^3_p (\Omega))}  \|\ve_1 - \ve_2\|_{X^1_T} \\
& \leq C R \left ( T^{ (1 - \frac{1}{p}) \theta}+T^{ (1 - \frac{1}{p}) \theta} \right ) \|\ve_1 - \ve_2\|_{X^1_T} .
\end{align*}
Analogously as before we can estimate the second summand by
\begin{align*}
 \| & \di (2 ((\eta (\varphi_1) - \eta (\varphi_2)) D\ve_2 ) ) \|_{Y^1_T} \\
& \leq 2 \|   \eta ' (\varphi_1) ( \nabla \varphi_1 -  \nabla \varphi_2) \cdot D\ve_2 \|_{Y^1_T} + 2 \| ( \eta ' (\varphi_1) - \eta ' (\varphi_2) ) \nabla \varphi_2 \cdot D\ve_2 \|_{Y^1_T} \\
& \ \ \  +  \|( \eta (\varphi_1) - \eta (\varphi_2) ) \Delta \ve_2\|_{Y^1_T} .
\end{align*}
For the sake of clarity we study these three terms separately again. 
Firstly,
\begin{align*}
 \|&   \eta ' (\varphi_1)  ( \nabla \varphi_1 -  \nabla \varphi_2) \cdot D\ve_2 \|_{Y^1_T}  \leq C (R) \left | \left | \|D\ve_2\|_{L^2 (\Omega)} \|\nabla \varphi_1 - \nabla \varphi_2\|_{C^1 (\overline \Omega)} \right | \right | _{L^2 (0,T)} \\
& \leq C (R) \underset{t \in [0,T]}{\sup} \|\nabla (\varphi_1 (t) - \varphi_2 (t)) - \nabla (\varphi_1 (0) - \varphi_2 (0))\|_{C^1 (\overline \Omega)} \|D\ve_2\|_{L^2 (0,T; L^2 (\Omega))} \\
& \leq C (R) T^{(1 - \frac{1}{p}) \theta} \|\nabla \varphi_1 - \nabla \varphi_2\|_{C^{(1 - \frac{1}{p}) \theta} ([0,T]; W^2_p (\Omega))} \|\ve_2\|_{X^1_T} \\
& \leq C(R) T^{(1 - \frac{1}{p}) \theta} \|\varphi_1 - \varphi_2\|_{X^2_T} ,
\end{align*}
where we used in the first step that $\eta ' (\varphi) $ is bounded in $ C([0,T]; C^2 (\overline \Omega))$.
Furthermore, (\ref{lipschitz_f_of_phi_w3p}) together with 
\begin{align*}
\varphi \in  C^{0, (1 - \frac{1}{p}) \theta} ([0,T]; W^3_p (\Omega)) \hookrightarrow C([0,T]; C^2 (\overline \Omega))
\end{align*} 
implies
\begin{align*}
\|  & ( \eta ' (\varphi_1)  - \eta ' (\varphi_2) ) \nabla \varphi_2 \cdot D\ve_2 \|_{Y^1_T} \\
& \leq \underset{t \in [0,T]}{\sup}\|  \eta ' (\varphi_1)  - \eta ' (\varphi_2) \|_{W^3_p (\Omega)}  \|\nabla \varphi_2\|_{C([0,T];C^1 (\overline \Omega))} \|D\ve_2\|_{L^2 (Q_T)} \\
& \leq C (R)  \underset{t \in [0,T]}{\sup} \|\varphi_1 (t) - \varphi_2 (t) \|_{W^3_p (\Omega)} 
 \leq C (R) T^{(1- \frac{1}{p}) \theta} \|\varphi_1 - \varphi_2\|_{X^2_T} 
\end{align*}
since $\varphi_1(0)-\varphi_2(0)=0$.
Analogously to the second summand we can estimate the third one by
\begin{align*}
 \|( \eta (\varphi_1) - \eta (\varphi_2) ) \Delta \ve_2\|_{Y_T} \leq C(R) T^{(1 - \frac{1}{p}) \theta} \|\varphi_1 - \varphi_2\|_{X^2_T} ,
\end{align*}
which shows the statement for the second term.

 For the third term we obtain
\begin{align*}
\|\rho & (\varphi_2)  \ve_2  \cdot \nabla \ve_2 - \rho(\varphi_1) \ve_1 \cdot \nabla \ve_1\|_{Y^1_T} \\
& \leq  \|( \rho(\varphi_2) - \rho (\varphi_1)) \ve_2 \cdot \nabla \ve_2 \|_{Y^1_T} + \| \rho(\varphi_1)  (\ve_2 - \ve_1) \cdot \nabla \ve_2 \|_{Y^1_T} \\
& \ \ \ + \| \rho (\varphi_1) \ve_1  \cdot (\nabla \ve_2 - \nabla \ve_1 ))\|_{Y^1_T} .
\end{align*}
We estimate these three terms separately again. For the first term we use that $\ve_2$ is bounded in $L^\infty (0,T; L^6 (\Omega))$, cf. (\ref{strong_solution_v_buc_h1}), and $\nabla \ve_2$ is bounded in $L^2 (0,T; L^6 (\Omega))$ together with (\ref{lipschitz_f_of_phi_w3p}). Thus
\begin{align*}
 \| & ( \rho(\varphi_2) -  \rho (\varphi_1)) \ve_2 \cdot \nabla \ve_2 \|_{Y^1_T}  \\
& \leq C (R) T^{(1- \frac{1}{p}) \theta} \|\varphi_2  - \varphi_1 \|_{C^{0, (1 - \frac{1}{p}) \theta} ([0,T]; W^3_p (\Omega))} \|\ve_2\|_{L^\infty (0,T; L^6 (\Omega))} \|\nabla \ve_2\|_{L^2 (0,T; L^6 (\Omega))} \\
& \leq C(R) T^{(1 - \frac{1}{p} ) \theta} \|\varphi_2 - \varphi_1\|_{X^1_T}.
\end{align*}
For the second term we use $\rho (\varphi_1) \in C([0,T]; C^2 (\overline \Omega))$, $\ve_i \in L^\infty (0,T; L^6 (\Omega))$ and $\nabla \ve_2 \in L^4 (0,T; L^3 (\Omega))$, cf. (\ref{strong_solution_v_buc_h1}) and (\ref{strong_solution_nabla_v_l4_l3}), $i = 1,2$. Hence,
\begin{align*}
 \| \rho(\varphi_1)  (\ve_2 - \ve_1) \cdot \nabla \ve_2 \|_{Y^1_T} & \leq  C (R) T^\frac{1}{4} \|\ve_1 - \ve_2\|_{L^\infty (0,T; L^6 (\Omega))} \| \nabla \ve_2 \|_{L^4 (0,T; L^3 (\Omega))} \\
& \leq C (R) T^\frac{1}{4}  \|\ve_1 - \ve_2\|_{X^1_T}.
\end{align*}
For the third term we use the same function spaces. This implies
\begin{align*}
 \| \rho (\varphi_1) \ve_1  \cdot (\nabla \ve_2 - \nabla \ve_1 ))\|_{Y_T} &\leq C (R) T^{\frac{1}{4}} \|\nabla \ve_1 - \nabla \ve_2\|_{L^4 (0,T; L^3 (\Omega))} \\
& \leq  C (R) T^{\frac{1}{4}} \|\ve_1 - \ve_2\|_{X^1_T} .
\end{align*}

Since $ \frac{\tilde \rho_1 - \tilde \rho_2}{2} $ is a constant, we obtain
\begin{align*}
& \left | \left |  \tfrac{\tilde \rho_1 - \tilde \rho_2}{2}  m(\varphi_1) \nabla ( \Delta \varphi_1 ) \cdot \nabla \ve_1 - \tfrac{\tilde \rho_1 - \tilde \rho_2}{2} m(\varphi_2) \nabla ( \Delta \varphi_2 ) \cdot \nabla \ve_2 \right | \right |_{Y^1_T} \\
& \ \ \  \leq C \left (  \|  m(\varphi_1) \nabla ( \Delta \varphi_1 ) \cdot ( \nabla \ve_1 - \nabla \ve_2 ) \|_{Y^1_T}  \right . \\
& \ \ \ \ \  \ +  \|   m(\varphi_1) ( \nabla (\Delta \varphi_1) - \nabla (\Delta \varphi_2)) \cdot \nabla \ve_2 \|_{Y^1_T}  \\
& \ \ \ \ \  \ + \left . \|  (m(\varphi_1) - m (\varphi_2) )\nabla (\Delta \varphi_2) \cdot \nabla \ve_2 \|_{Y^1_T} \right ) .
\end{align*}
For the sake of clarity we study all three terms separately again. In the following we use $\nabla \Delta \varphi_i \in L^\infty (0,T; L^4 (\Omega))$, cf. (\ref{phi_hoelder_continuous_in_time_w3p}), $\nabla \ve_i \in L^{\frac{8}{3}} (0,T; L^4 (\Omega))$, cf. (\ref{strong_solution_nabla_v_l83_l4}), for $i = 1,2$, and $m (\varphi_1) \in C([0,T];C^2 (\overline \Omega))$. Altogether this implies
\begin{align*}
 \|  & m(\varphi_1) \nabla ( \Delta \varphi_1 ) \cdot ( \nabla \ve_1 - \nabla \ve_2 ) \|_{Y^1_T}  \\
 &\leq C T^\frac{1}{8} \|\nabla \Delta \varphi_1\|_{L^\infty (0,T; L^4 (\Omega))} \|\nabla \ve_1 - \nabla \ve_2\|_{L^{\frac{8}{3}} (0,T; L^4 (\Omega))} \\
& \leq C (R) T^\frac{1}{8}  \| \ve_1 - \ve_2\|_{X^1_T} .
\end{align*}
Analogously the second summand yields
\begin{align*}
 \|   m(\varphi_1) ( \nabla (\Delta \varphi_1) - \nabla (\Delta \varphi_2)) \cdot \nabla \ve_2 \|_{Y^1_T} & \leq C (R) T^\frac{1}{8} \|\varphi_1 - \varphi_2\|_{X^2_T} .
\end{align*}
For the last term we use $m (\varphi_i) \in C^{0, (1 - \frac{1}{p}) \theta} ([0,T]; W^3_p (\Omega)) \hookrightarrow C^0 ([0,T]; C^2 (\overline \Omega))$ together with (\ref{lipschitz_f_of_phi_w3p}) and obtain
\begin{align*}
 |&|  (m(\varphi_1) - m (\varphi_2) )\nabla (\Delta \varphi_2) \cdot \nabla \ve_2 \|_{Y^1_T} \\
& \leq C(R ) T^\frac{1}{8} \|\varphi_1 (t) - \varphi_2 (t)\|_{C^0([0,T]; C^2 (\overline \Omega))} \|\nabla \Delta \varphi_2\|_{L^\infty (0,T; L^4 (\Omega)} \|\nabla \ve_2\|_{L^\frac{8}{3} (0,T; L^4 (\Omega))} \\
& \leq  C(R) T^{\frac{1}{8} } \|\varphi_1 - \varphi_2\|_{X^2_T} .
\end{align*}

The next term has the same structure as the one before and can be estimated as
\begin{align}\label{strong_solution_nabla_w_prime}
& \left | \left |   \tfrac{\tilde \rho_1 - \tilde \rho_2}{2}  m(\varphi_1) \nabla ( W' (\varphi_1) ) \cdot \nabla \ve_1 - \tfrac{\tilde \rho_1 - \tilde \rho_2}{2} m(\varphi_2) \nabla (  W' (\varphi_2) ) \cdot \nabla \ve_2 \right | \right |_{Y^1_T}  \nonumber \\
& \leq C \left (  \|  m(\varphi_1) \nabla  W' ( \varphi_1) \cdot ( \nabla \ve_1 - \nabla \ve_2 ) \|_{Y^1_T}  \right . \nonumber \\
&  \ \ \ +  \|   m(\varphi_1) ( \nabla  W'( \varphi_1 ) - \nabla W' ( \varphi_2)) \cdot \nabla \ve_2 \|_{Y^1_T}   \nonumber \\
& \ \ \ + \left . \|  (m(\varphi_1) - m (\varphi_2) )\nabla W' (\varphi_2 ) \cdot \nabla \ve_2 \|_{Y^1_T} \right ) .
\end{align}
For $\nabla \ve_i$, $i = 1,2$, we use its boundedness in $L^4 (0,T; L^3 (\Omega))$, cf. (\ref{strong_solution_nabla_v_l4_l3}). Moreover, we know $\nabla W' (\varphi) \in C([0,T]; W^{3 - \frac{4}{p}}_p (\Omega))$ and $m ( \varphi) \in C([0,T]; C^2 (\overline \Omega))$ for $\varphi \in B_R^{X^2_T}$. 
Using all these bounds we can estimate the three terms in (\ref{strong_solution_nabla_w_prime}) separately. For the first term we obtain
\begin{align*}
 \|  m(\varphi_1) \nabla  W' ( \varphi_1) \cdot ( \nabla \ve_1 - \nabla \ve_2 ) \|_{Y^1_T}  & \leq C (R) T^\frac{1}{4} \|\nabla \ve_1 - \nabla \ve_2\|_{L^4 (0,T; L^3 (\Omega))}  \\
& \leq C (R) T^\frac{1}{4} \|\ve_1 - \ve_2\|_{X^1_T} .
\end{align*}

For the second summand in (\ref{strong_solution_nabla_w_prime}) we have to estimate the difference \linebreak $\nabla W' (\varphi_1) - \nabla W' (\varphi_2)$ in an appropriate manner. To this end, we use (\ref{phi_hoelder_continuous_in_time_w3p}), (\ref{lipschitz_f_of_phi_w3p}) and $W^2_p (\Omega) \hookrightarrow C^1 (\overline \Omega)$. Moreover, we use $\nabla \ve_2 \in L^4 (0,T; L^3 (\Omega))$, cf. (\ref{strong_solution_nabla_v_l4_l3}), and $m (\varphi ) \in C([0,T]; C^2 (\overline \Omega))$. Then it follows
\begin{align*}
\|  m(\varphi_1) & ( \nabla  W'( \varphi_1 ) - \nabla W' ( \varphi_2)) \cdot \nabla \ve_2 \|_{Y^1_T} \\
& \leq C (R) T^\frac{1}{4} \underset{t \in [0,T]}{\sup} \|\nabla W' (\varphi_1 (t) ) - \nabla W' (\varphi_2 (t))\|_{W^2_p (\Omega)} \\
& \leq  C (R) T^\frac{1}{4}  \underset{t \in [0,T]}{\sup} \|\varphi_1 (t) - \varphi_2 (t)\|_{W^3_p (\Omega)} \\
& \leq C (R)  T^{\frac{1}{4} + (1 - \frac{1}{p}) \theta } \|  \varphi_1 - \varphi_2 \|_{X^2_T} .
\end{align*}
So it remains to estimate the third term of  (\ref{strong_solution_nabla_w_prime}). As before we get
\begin{align*}
\|  & (m(\varphi_1) - m (\varphi_2) )\nabla W' (\varphi_2 ) \cdot \nabla \ve_2 \|_{Y^1_T} \\
& \leq C(R) T^{\frac{1}{4} + (1 - \frac{1}{p}) \theta} \|\varphi_1 - \varphi_2\|_{C^{0, (1 - \frac{1}{p} ) \theta} ([0,T]; W^3_p (\Omega))} \\
& \ \ \ \ \|\nabla W' (\varphi_2)\|_{BUC([0,T]; C^1 (\overline \Omega))}  \|\nabla \ve_2\|_{L^4 (0,T; L^3 (\Omega))} \\
& \leq  C(R) T^{\frac{1}{4} + (1 - \frac{1}{p}) \theta} \|\varphi_1 - \varphi_2\|_{X^2_T} ,
\end{align*}
which completes the estimate for (\ref{strong_solution_nabla_w_prime}).

 Finally, we study the last term of $ \| \mathbb P_\sigma ( F_1 (\ve_1, \varphi_1 ) - F_1 (\ve_2, \varphi_2 ) )\|_{L^2 (Q_T)}$. It holds
\begin{align*}
\| \Delta \varphi_2 \nabla \varphi_2 -  \Delta \varphi_1 \nabla \varphi_1\|_{Y_T} \leq \| \Delta \varphi_2 ( \nabla \varphi_2 - \nabla \varphi_1) \|_{Y_T} + \| ( \Delta \varphi_2 - \Delta \varphi_1) \nabla \varphi_1 \|_{Y_T}.
\end{align*}
Using $\Delta \varphi_i \in C([0,T]; C^0 (\overline \Omega))$ and $\nabla \varphi_i \in C^{0, (1- \frac{1}{p} ) \theta} ([0,T]; W^2_p (\Omega))$, $i = 1,2$, cf. (\ref{phi_hoelder_continuous_in_time_w3p}), the first term can be estimated by
\begin{align*}
\| \Delta \varphi_2 ( \nabla \varphi_2 - \nabla \varphi_1) \|_{Y^1_T} & \leq C (R) T^{\frac{1}{2} + (1 - \frac{1}{p}) \theta)} \|\nabla \varphi_1 - \nabla \varphi_2\|_{C^{0, ( 1 - \frac{1}{p}) \theta}([0,T]; W^2_p (\Omega))} \\
& \leq C (R) T^{\frac{1}{2} + (1 - \frac{1}{p}) \theta)} \| \varphi_1 - \varphi_2\|_{X^2_T} .
\end{align*}
Analogously the second term can be estimated by
\begin{align*}
 \| ( \Delta \varphi_2 - \Delta \varphi_1) \nabla \varphi_1 \|_{Y_T} \leq C (R) T^{\frac{1}{2} + (1 - \frac{1}{p}) \theta)} \| \varphi_1 - \varphi_2\|_{X^2_T} .
\end{align*}
Hence, we obtain
\begin{align*}
 \| \mathbb P_\sigma ( F_1 (\ve_1, \varphi_1 ) - F_1 (\ve_2, \varphi_2 ) )\|_{L^2 (Q_T)} & \leq C(R,T) \|(\ve_1 - \ve_2), (\varphi_1 - \varphi_2)\|_{X_T}
\end{align*}
for a constant $C(R,T) > 0$ such that $C(R,T) \rightarrow 0$ as $T \rightarrow 0$.

Remember that we study the nonlinear operator $\mathcal F \colon X_T \rightarrow Y_T$ given by
\begin{align*}
\mathcal F (\ve , \varphi) = 
\begin{pmatrix}
\mathbb P_\sigma F_1 (\ve, \varphi)  \\
- \nabla \varphi \cdot \ve + \di ( \frac{1}{\varepsilon} m(\varphi) \nabla W' (\varphi)) + \varepsilon  m (\varphi_0) \Delta^2 \varphi - \varepsilon \di ( m (\varphi) \nabla \Delta \varphi)
\end{pmatrix} 
\end{align*}
and we want to show its Lipschitz continuity such that (\ref{strong_solution_estimate_lipschitz_continuity_F}) holds. We already showed its Lipschitz continuity for the first part. Now we continue to study the second one. This part has to be estimated in $L^p (0,T; L^p (\Omega))$ for $4 < p < 6$.

For the analysis we use the boundedness of $\nabla \varphi $ in $ C([0,T]; C^1 (\overline \Omega))$ and of $\ve$ in $ L^\infty (0,T; L^6 (\Omega))$.
Then it holds 
\begin{align*}
 \| & (\nabla \varphi_1 \cdot \ve_1  - \nabla \varphi_2 \cdot \ve_2 ) \|_{L^p (Q_T)} \\
& \leq \|\nabla \varphi_1 \cdot ( \ve_1 - \ve_2 )\|_{L^p (Q_T)}   + \|  ( \nabla \varphi_1 - \nabla \varphi_2 ) \cdot \ve_2\|_{L^p (Q_T)}   \\
& \leq T^\frac{1}{p} \|\nabla \varphi_1\|_{L^\infty (Q_T)} \|\ve_1 - \ve_2\|_{L^\infty (0,T; L^6 (\Omega))} \\
& \ \ \ + T^\frac{1}{p} \|\nabla \varphi_1 - \nabla \varphi_2\|_{L^\infty (Q_T)}  \|\ve_2\|_{L^\infty (0,T; L^6 (\Omega))} \\
& \leq T^\frac{1}{p} R \|\ve_1 - \ve_2\|_{X^1_T} + T^\frac{1}{p} R \|\varphi_1 - \varphi_2\|_{X^2_T} .
\end{align*}
Next we study the term $\di (m (\varphi) \nabla W' (\varphi))$. We use the boundedness of $f(\varphi) $ in $C([0,T]; C^2 (\overline \Omega)) \cap C^{0, (1 - \frac{1}{p}) \theta} ([0,T]; W^3_p (\Omega))$ for $f \in \{ m , W' \}$ and $\varphi \in X^2_T $ with \linebreak $\|\varphi\|_{X^2_T} \leq R$. Then it holds
\begin{align*}
\|\di ( & m(\varphi_1) \nabla W' (\varphi_1)) - \di ( m (\varphi_2) \nabla W' (\varphi_2))\|_{Y^2_T} \\
& \leq C(R) \|m (\varphi_1) \nabla W' (\varphi_1)) -  m (\varphi_2) \nabla W' (\varphi_2)\|_{L^p (0,T; W^1_p (\Omega))} \\
& \leq C(R) T^\frac{1}{p} \underset{t \in [0,T]}{\sup} \|m (\varphi_1 (t)) - m (\varphi_2 (t))\|_{W^3_p (\Omega)} \|\nabla W' (\varphi_1)\|_{C([0,T]; C^1 (\overline \Omega))} \\
& \ \ \ + C(R) T^\frac{1}{p} \|m (\varphi_2)\|_{C([0,T]; C^2 (\overline \Omega))} \underset{t \in [0,T]}{\sup} \| W' (\varphi_1 (t)) - W' (\varphi_2 (t)) \|_{W^3_p (\Omega)} \\
& \leq C(R) T^\frac{1}{p} \left ( \underset{t \in [0,T]}{\sup} \|\varphi_1 (t) - \varphi_2 (t)\|_{W^3_p (\Omega)} + \underset{t \in [0,T]}{\sup} \| \varphi_1 (t) - \varphi_2 (t) \|_{W^3_p (\Omega)} \right ) \\
&\leq C(R) T^\frac{1}{p}  \underset{t \in [0,T]}{\sup} \|(\varphi_1 (t) - \varphi_2 (t)) - (\varphi_1 (0) - \varphi_2 (0))\|_{W^3_p (\Omega)} \\
& \leq C(R) T^{\frac{1}{p} + (1 - \frac{1}{p} ) \theta} \|\varphi_1 - \varphi_2\|_{C^{0, (1 - \frac{1}{p}) \theta} ([0,T]; W^3_p (\Omega))} \\
& \leq C(R) T^{\frac{1}{p} + (1 - \frac{1}{p} ) \theta} \|\varphi_1 - \varphi_2\|_{X^2_T} .
\end{align*}
Here we also used $\varphi_1 (0) = \varphi_2 (0) = \varphi_0$ for $\varphi_1, \varphi_2 \in X^2_T$ and (\ref{lipschitz_f_of_phi_w3p}).

Now there remain two terms which we need to study together for the proof of the Lipschitz continuity. Due to the boundedness of $m (\varphi) $ in $ BUC([0,T]; W^3_p (\Omega))$ and of $\nabla \Delta \varphi$ in $ L^p (0,T; W^1_p (\Omega))$, Theorem \ref{theorem_composition_sobolev_functions} yields the boundedness of $m (\varphi) \nabla \Delta \varphi$ in $ L^p (0,T; W^1_p (\Omega))$. Hence, this term is well-defined in the $L^p (Q_T)$-norm. We omit the prefactor $\varepsilon $ for both terms again and estimate
\begin{align}
&\|m (\varphi_0)  \Delta^2 \varphi_1 -  m (\varphi_0) \Delta^2 \varphi_2 + \di (m (\varphi_2) \nabla \Delta \varphi_2) - \di (m (\varphi_1) \nabla \Delta \varphi_1)\|_{L^p (Q_T)} \nonumber \\
& = \| (m(\varphi_0) - m (\varphi_1)) ( \Delta^2 \varphi_1 - \Delta^2 \varphi_2) + m (\varphi_1) \Delta^2 \varphi_1 - m (\varphi_1) \Delta^2 \varphi_2 + \nabla m (\varphi_2) \cdot \nabla \Delta \varphi_2 \nonumber \\
& \ \ \  + m ( \varphi_2) \Delta^2 \varphi_2 - \nabla m (\varphi_1) \cdot \nabla \Delta \varphi_1 - m (\varphi_1) \Delta^2 \varphi_1   \|_{L^p (Q_T)} \nonumber \\
& \leq \| ( m ( \varphi_1 (0) - m (\varphi_1)) (\Delta^2 \varphi_1 - \Delta^2 \varphi_2 )\|_{L^p (Q_T)} + \|  (m (\varphi_2) - m (\varphi_1)) \Delta^2 \varphi_2  \|_{L^p (Q_T)} \nonumber \\
& \ \ \ + \| \nabla  m (\varphi_2) \cdot \nabla \Delta \varphi_2 - \nabla m ( \varphi_1) \cdot \nabla \Delta \varphi_1 \|_{L^p (Q_T)}  \label{complicated_terms_lipschitz_continuity}
\end{align}
For the sake of clarity, we estimate these three terms separately again. Due to the boundedness of $m (\varphi_1)$ in $ C^{0, (1 - \frac{1}{p}) \theta} ([0,T] ; W^3_p (\Omega))$ we obtain for the first term
\begin{align*}
 \| ( m ( & \varphi_1 (0) - m (\varphi_1))  (\Delta^2 \varphi_1 - \Delta^2 \varphi_2 )\|_{L^p (Q_T)} \\
 &  \leq \underset{t \in (0,T)}{\sup} \|  m(\varphi_1 (0)) - m ( \varphi_1 (t)) \|_{C^0 (\overline \Omega)} \|\Delta^2 \varphi_1 - \Delta^2 \varphi_2\|_{L^p (Q_T)} \\
& \leq C(R) T^{(1 - \frac{1}{p}) \theta} \|m (\varphi_1)\|_{C^{0, (1-\frac{1}{p}) \theta } ([0,T]; W^3_p (\Omega))} \|\varphi_1 - \varphi_2\|_{X^2_T} .
\end{align*}
Since $m (\varphi_1)$ is bounded in $ C^{0, (1 - \frac{1}{p}) \theta} ([0,T] ; W^3_p (\Omega))$, we can estimate the second term in (\ref{complicated_terms_lipschitz_continuity}) by
\begin{align*}
\|  (m (\varphi_2) & - m (\varphi_1)) \Delta^2 \varphi_2  \|_{L^p (Q_T)}  \leq \underset{t \in (0,T)}{\sup}\|m (\varphi_2 (t)) - m (\varphi_1 (t))\|_{C^2 (\overline \Omega)} \|\Delta^2 \varphi_2\|_{L^p (Q_T)}  \\
& \leq C(R) \underset{t \in (0,T)}{\sup}\|m (\varphi_2 (t)) - m (\varphi_1 (t))\|_{W^3_p (\Omega)}   \\
& \leq C(R) \underset{t \in (0,T)}{\sup}\|(\varphi_2 (t) - \varphi_1 (t)) - (\varphi_2(0) - \varphi_1 (0))\|_{W^3_p (\Omega)}  \\
& \leq C(R) T^{(1 - \frac{1}{p}) \theta } \|\varphi_1 - \varphi_2\|_{C^{0, ( 1 - \frac{1}{p}) \theta} ([0,T]; W^3_p (\overline \Omega))}  ,
\end{align*}
where we used (\ref{lipschitz_f_of_phi_w3p}) again in the penultimate step. Finally, we study the last term in (\ref{complicated_terms_lipschitz_continuity}). Here we get
\begin{align}\label{strong_solution_lipschitz_continuity_most_complicated_term_}
\| & \nabla  m (\varphi_2) \cdot \nabla \Delta \varphi_2 - \nabla m ( \varphi_1) \cdot \nabla \Delta \varphi_1 \|_{L^p (Q_T)} \nonumber \\
& \leq \|(\nabla m (\varphi_2) - \nabla m (\varphi_1)) \cdot \nabla \Delta \varphi_2 \|_{L^p (Q_T)} \nonumber \\
& \ \ \ + \|\nabla m (\varphi_1) \cdot ( \nabla \Delta \varphi_2  - \nabla \Delta \varphi_1)\|_{L^p (Q_T)} .
\end{align}
Since $\nabla m (\varphi_1) $ is bounded in $C([0,T]; C^1 (\overline \Omega))$ and $\nabla \Delta \varphi_i$ is bounded in \linebreak $C([0,T]; L^p (\Omega))$ for $i = 1,2$, we can estimate the second summand by
\begin{align*}
\|\nabla m ( & \varphi_1)  \cdot ( \nabla \Delta \varphi_2  - \nabla \Delta \varphi_1)\|_{L^p (Q_T)} \\
&\leq C(R)  T^{\frac{1}{p}} \|\nabla m (\varphi_1)\|_{C ([0,T]; C^1 (\overline \Omega))} \|\nabla \Delta \varphi_1 - \nabla \Delta \varphi_2\|_{C([0,T]; L^p (\Omega))} \\
& \leq C(R) T^\frac{1}{p} \|\varphi_1 - \varphi_2\|_{X^2_T} .
\end{align*}
Thus it remains to estimate the first term of (\ref{strong_solution_lipschitz_continuity_most_complicated_term_}). Here we get
\begin{align*}
 \|(\nabla m & (\varphi_2) - \nabla m (\varphi_1)) \cdot \nabla \Delta \varphi_2 \|_{L^p (Q_T)} \\
& \leq C(R) T^\frac{1}{p} \underset{t \in [0,T]}{\sup} \|\nabla m (\varphi_2(t)) - \nabla m (\varphi_1 (t))\|_{C^0 (\overline \Omega)} \|\nabla \Delta \varphi_2\|_{C([0,T]; L^p (\Omega))} \\
& \leq C(R) T^\frac{1}{p} \underset{t \in [0,T]}{\sup} \|m (\varphi_2 (t)) - m (\varphi_1 (t))\|_{W^3_p (\Omega)} \|\varphi_2\|_{C([0,T] ; W^3_p (\Omega))} \\
& \leq C(R) T^\frac{1}{p}  \underset{t \in [0,T]}{\sup} \|\varphi_1 (t) - \varphi_2 (t)\|_{W^3_p (\Omega)} \\
& \leq C(R) T^{\frac{1}{p} + (1 - \frac{1}{p}) \theta} \|\varphi_1 - \varphi_2\|_{C^{0, (1 - \frac{1}{p}) \theta} ([0,T] ; W^3_p (\Omega))} .
\end{align*}
Hence, (\ref{strong_solution_lipschitz_continuity_most_complicated_term_}) is Lipschitz continuous and therefore also the second part of $\mathcal F$ is Lipschitz continuous. Together with the Lipschitz continuity of the first part of $\mathcal F$ we have shown 
\begin{align*}
\| \mathcal F (\ve_1 , \varphi_1 ) - \mathcal F (\ve_2 , \varphi_2) \| _{Y_T} \leq C (T, R) \|(\ve_1 - \ve_2, \varphi_1 - \varphi_2 )\|_{X_T}
\end{align*}
for all $(\ve_i, \varphi_i) \in X_T$ with $\|(\ve_i, \varphi_i)\|_{X_T} \leq R$, $i = 1,2$, and a constant $C(T,R) > 0 $ such that $C (T,R) \rightarrow 0$ as $T \rightarrow 0$.

\section{Existence and Continuity of $\mathcal L^{-1}$ }\label{sec:Linear}

In the following we need:
\begin{theorem}\label{strong_solution_showalter_theorem_6_1}
Let the linear, symmetric and monotone operator $\mathcal B$ be given from the real vector space $E$ to its algebraic dual $E'$, and let $E' _b$ be the Hilbert space which is the dual of $E$ with the seminorm 
\begin{align*}
|x|_b = \mathcal B x (x) ^\frac{1}{2}, \qquad x \in E .
\end{align*}
Let $A \subseteq E \times E' _b$ be a relation with domain $D = \{  x \in E: \ A(x) \neq \emptyset  \}$. Let $A$ be the subdifferential, $\partial \varphi$, of a convex lower-semi-continuous function $\varphi: E_b \rightarrow [0, \infty ]$ with $\varphi (0) = 0$. Then for each $u_0$ in the $E_b$-closure of $\mathrm{dom} (\varphi)$ and each $f \in L^2 (0,T; E' _b)$ there is a solution  $u : [0,T] \rightarrow E$ with $\mathcal B u \in C([0,T], E' _b)$ of
\begin{align*}
\frac{d}{dt} ( \mathcal B u (t) ) + A ( u (t)) \ni f(t) , \qquad 0 < t < T, 
\end{align*}
with
\begin{align*}
\varphi \circ u \in L^1 (0,T), \sqrt{t} \frac{d}{dt} \mathcal B u (\cdot ) \in L^2 (0,T; E' _b), u(t) \in D, \text{ a.e. } t \in [0,T] ,
\end{align*}
and $\mathcal B u(0) = \mathcal B u_0$. If in addition $u_0 \in \mathrm{dom} (\varphi)$, then 
\begin{align*}
\varphi \circ u \in L^\infty (0,T), \qquad \frac{d}{dt} \mathcal B u \in L^2 (0,T; E' _b) .
\end{align*}
\end{theorem}

The proof of this theorem can be found in \cite[Chapter IV, Theorem 6.1]{Showalter}.

To prove Theorem~\ref{thm:linear} we need to show the existence of $(\tilde \ve, \tilde \varphi) \in X_T$ such that (\ref{strong_solution_tilde_v0_tilde_phi0_bounded}) holds and to prove that $\mathcal L : X_T \rightarrow Y_T$ is invertible with uniformly bounded inverse, i.e., there exists a constant $C > 0$ which does not depend on $T$ such that $\|\mathcal L^{-1}\|_{\mathcal L (Y_T, X_T)} \leq C$. Recall that the linear operator $\mathcal L \colon X_T \rightarrow Y_T$ is defined by
\begin{align*}
\mathcal L (\ve, \varphi) = 
\begin{pmatrix}
\mathbb P_\sigma ( \rho_0 \partial_t \ve ) - \mathbb P_\sigma ( \di (2 \eta (\varphi_0) D\ve ))  \\
\partial_t \varphi + \varepsilon m (\varphi_0) \Delta^2 \varphi
\end{pmatrix} .
\end{align*}
We note that the first part only depends on $\ve$ while the second part only depends on $\varphi$. Thus both equations can be solved separately.

To show the existence of a unique solution $\ve$ for every right-hand side $\bold f$ in the first equation we use Theorem \ref{strong_solution_showalter_theorem_6_1}.

So we have to specify what $E$, $E ' _b$, $\varphi$ and so on are in the problem we study and  show that the conditions of Theorem \ref{strong_solution_showalter_theorem_6_1} are fulfilled. Then Theorem~\ref{strong_solution_showalter_theorem_6_1} yields the existence of a solution. More precisely, we obtain the following lemma.

\begin{lemma}\label{strong_solution_existence_proof_L_1_part1}
Let Assumption \ref{strong_solutions_general_assumptions} hold. Then for every $\bold v_0 \in H^1_0 (\Omega)^d \cap L^2_\sigma (\Omega)$, \linebreak $f \in L^2 (0,T; L^2_\sigma (\Omega))$, $\varphi_0 \in W^{1}_r (\Omega)$, $r > d \geq 2$, and every $0 < T < \infty$ there \linebreak exists a unique solution
\begin{align*} 
\bold v\in W^1_2 (0,T; L^2_\sigma (\Omega)) \cap L^\infty (0,T; H^1_0 (\Omega)^d)
\end{align*} 
such that 
\begin{align}
\mathbb P_\sigma (\rho_0 \partial_t \bold v) - \mathbb P_\sigma (\mathrm{div} (2 \eta (\varphi_0) D \bold v)) &= f   && \text{ in } Q_T,  \label{strong_solution_L_first_equation} \\
\mathrm{div} (\bold v) &= 0 && \text{ in } Q_T , \\
\bold v_{|\partial \Omega} &= 0 && \text{ on } (0, T) \times  \partial \Omega  , \\
\bold v (0) &= \bold v_0  && \text{ in } \Omega  \label{strong_solution_L_first_equation_initial_data}
\end{align}
for a.e. $(t,x)$ in $(0,T) \times \Omega$, where $\bold v (t) \in H^2 (\Omega)^d$ for a.e. $t \in (0,T)$.
\end{lemma}

\begin{proof} 
Since we want to solve (\ref{strong_solution_L_first_equation})-(\ref{strong_solution_L_first_equation_initial_data}) with Theorem \ref{strong_solution_showalter_theorem_6_1}, we define
\begin{align*}
\mathcal B u := \mathbb P _\sigma ( \rho_0 u)  
\end{align*}
for $u \in E$, where we still need to specify the real vector space $E$.
But as we want to have $\frac{d}{dt} \mathcal B u \in L^2 (0,T; L^2_\sigma (\Omega))$, the dual space $E'_b$ has to coincide with $L^2_\sigma (\Omega)$. But this can be realized by choosing $E = L^2_\sigma (\Omega)$. Then $E_b' \cong L^2_\sigma (\Omega)$ and with the notation in Theorem \ref{strong_solution_showalter_theorem_6_1} we get the Hilbert space $E'_b$ equipped with the seminorm
\begin{align*}
|\ue|_b = \mathcal B \bold \ue ( \ue) ^\frac{1}{2} &= \left ( \into \mathbb P _\sigma ( \rho_0 \ue) \cdot \ue  \dx \right ) ^\frac{1}{2} = \left ( \into \rho_0 \ue \cdot \mathbb P_\sigma \ue \dx \right ) ^\frac{1}{2} \\
&= \left ( \into \rho_0 |\ue|^2 \dx \right ) ^\frac{1}{2} \cong \|\ue\|_{L^2 (\Omega)} .
\end{align*}
Thus we obtain $E '_b \cong L^2_\sigma (\Omega) = E_b$.
Moreover, we define $A: \mathcal D (A) \rightarrow  L^2_\sigma (\Omega) ' \cong L^2_\sigma (\Omega)$ by 
\begin{align}\label{strong_solution_definition_A}
(A \ue )(\ve) :=
\begin{cases} 
-  \into  \mathbb P _\sigma \di (2 \eta (\varphi_0) D\ue ) \cdot \ve \dx & \text{ if } \ue \in \mathrm{dom} (A) \\
\emptyset & \text{ if } \ue \notin \mathrm{dom} (A) .
\end{cases}
\end{align}
for every $ \ve \in L^2_\sigma (\Omega)$ and  $\mathcal D (A) = H^2 (\Omega)^d \cap H^1_0 (\Omega)^d \cap L^2_\sigma (\Omega)$. Thus we get for the relation $\mathcal A$ defined by $ \mathcal A := \{  (\ue,\ve) : \ \ve = A\ue , \ \ue \in \mathcal D (A) \} $ the following inclusions
\begin{align*}
\mathcal A  = \{ (\ue , - \mathbb P_\sigma \di (2 \eta (\varphi_0) D\ue) :  \ue \in H^2 (\Omega)^d \cap H^1_0 (\Omega)^d \cap L^2_\sigma (\Omega) \}  \subseteq  E \times E'_b ,
\end{align*}
since the term  $ \mathbb P_\sigma \di (2 \eta (\varphi_0) D\ue)$ is in $L^2 _\sigma (\Omega) ' \cong L^2_\sigma (\Omega)$. Now we define \linebreak $\psi : L^2_\sigma (\Omega) \rightarrow [0, + \infty]$ by
\begin{align}\label{strong_solution_definition_phi}
\psi (\ue) := 
\begin{cases}
\into  \eta (\varphi_0) D\ue : D\ue \ dx & \text{ if } \ue \in H^1_0 (\Omega)^d \cap L^2_\sigma (\Omega) = \mathrm{dom} (\psi)  , \\
+ \infty &  \text{ else. }
\end{cases}
\end{align}
We note $\psi (0) = 0$ and $\ve_0$ is in the $L^2$-closure of $\mathrm{dom}(\psi)$, i.e., in $L^2_\sigma (\Omega)$. 
Hence, it remains to show that $\psi$ is convex and lower-semi-continuous and that $A$ is the subdifferential of $\psi$. Then we can apply Theorem \ref{strong_solution_showalter_theorem_6_1}. But the first two properties are obvious. Thus it remains to show the subdifferential property, which is satisfied by Lemma \ref{strong_solution_A_coincides_partial_phi} below. Hence, we are able to apply Theorem \ref{strong_solution_showalter_theorem_6_1} which yields the existence.
Moreover, the initial condition is also fulfilled as Theorem \ref{strong_solution_showalter_theorem_6_1} yields 
\begin{align*}
\mathbb P_\sigma (\rho_0 \ve(0)) = \mathcal B \ve (0) = \mathcal B \ve_0 = \mathbb P_\sigma (\rho_0 \ve_0)  \qquad \text{ in } L^2 (\Omega) .
\end{align*}
In particular we can conclude
\begin{align*}
0 &= \int \limits_\Omega \mathbb P_\sigma ( \rho_0 \ve(0) - \rho_0 \ve_0 ) \cdot \boldsymbol \psi \mathit {dx} = \int  \limits_\Omega (\rho_0 \ve(0) - \rho_0 \ve_0) \cdot \boldsymbol \psi \mathit{dx}
\end{align*}
for every $\boldsymbol \psi \in C^\infty_{0, \sigma} (\Omega)$. By approximation this identity also holds for $\boldsymbol \psi := \ve(0) - \ve_0 \in L^2_\sigma (\Omega)$ and we get 
\begin{align*}
\int \limits_\Omega \rho_0 |\ve(0) - \ve_0|^2 \mathit{dx} = 0 .
\end{align*}
This implies $\ve(0) = \ve_0$ in $L^2_\sigma (\Omega)$.

For the uniqueness we consider $\ve_1 , \ve_2 \in W^1_2 (0,T; L^2_\sigma (\Omega)) \cap L^\infty (0,T; H^1_0 (\Omega) \cap L^2_\sigma (\Omega))$ such that (\ref{strong_solution_L_first_equation}) holds for a.e. $(t,x) \in (0,T) \times \Omega$. Then $\ve := \ve_1 - \ve_2$ solves the homogeneous equation a.e. in $(0,T) \times \Omega$. Testing this homogeneous equation with $\ve$ we get
\begin{align*}
 \into  \frac{1}{2} \rho_0   \ve^2_{|t=T} \dx  + \int \limits_0^T \into 2 \eta (\varphi_0) D\ve : D\ve \dx \dt = 0.
\end{align*}
Hence, it follows $\ve \equiv 0$ and therefore $\ve_1 = \ve_2$, which yields the uniqueness.
\end{proof}

In the proof above we used that the mapping $A$ coincides with the subdifferential $\partial \varphi$. More precisely, we have the following lemma.

\begin{lemma}\label{strong_solution_A_coincides_partial_phi}
Let $\Omega \subseteq \mathbb R^d$, $d = 2,3$, be a domain and $\psi\colon L^2_\sigma (\Omega) \rightarrow [0, + \infty]$ be given as in (\ref{strong_solution_definition_phi}). Moreover, we consider $A : L^2_\sigma (\Omega) \rightarrow  L^2_\sigma (\Omega)$ to be given as in (\ref{strong_solution_definition_A}). Then it holds
\begin{enumerate}
\item $ \mathcal D (\partial \psi) = \mathcal D (A) . $
\item $\partial \psi (\ue) = \{ A \ue \} \text{ for all } \ue \in \mathcal D (A) .$
\end{enumerate}
\end{lemma}

\begin{proof}
Remember that 
\begin{align*}
\mathcal D (\partial \psi) = \{  \ve \in L^2_\sigma (\Omega) : \ \partial \psi (\ve) \neq \emptyset \}
\end{align*}
and $\mathcal D (A) = H^2 (\Omega)^d \cap H^1_0 (\Omega)^d \cap L^2_\sigma (\Omega)$ by definition.

\textbf{$\bold{1^{st}}$ part: } $\mathcal D (A) \subseteq \mathcal D (\partial \psi)$ and $A\ue \in \partial \psi (\ue)$ for every $\ue \in \mathcal D (A)$. 
\\
To show the first part of the proof let $\ue \in \mathcal D (A)$ be given. If it holds $\ve \in L^2_\sigma (\Omega)$ but $\ve \notin H^1_0 (\Omega)^d$, then the inequality
\begin{align*}
\weight{A\ue, \ve-\ue}_{L^2 (\Omega)} \leq \psi (\ve) - \psi (\ue)
\end{align*} 
is satisfied since it holds $\psi (\ve) = + \infty$ in this case by definition. 
\\
So let $\ve \in H^1_0 (\Omega)^d \cap L^2_\sigma (\Omega)$. Then it holds
\begin{align*}
\weight{A\ue, \ve - \ue}_{L^2 (\Omega)} &= - \into \mathbb P_\sigma \di (2 \eta (\varphi_0) D\ue) \cdot (\ve- \ue) \dx \\
& =  \into 2 \eta (\varphi_0) D\ue : D\ve \dx - \into 2 \eta (\varphi_0) D\ue :  D \ue \dx \\
& \leq \into \eta (\varphi_0) |D\ue|^2 \dx + \into \eta (\varphi_0) |D\ve|^2 \dx - 2 \into \eta (\varphi_0) |D\ue|^2 \dx \\
& = \psi (\ve) - \psi (\ue) 
\end{align*}
for every $\ve \in H^1_0 (\Omega)^d \cap  L^2_\sigma (\Omega)$.
This implies that $A\ue$ is a subgradient of $\psi$ at $\ue$, i.e., $ A\ue \in \partial \psi (\ue)$, and $\partial \psi (\ue) \neq \emptyset$, i.e., $\ue \in \mathcal D (A) \subseteq \mathcal D (\partial \psi)$. Hence, we have shown the first part of the proof.

\textbf{$\bold{2^{nd}}$ part: } $\mathcal D (\partial \psi) \subseteq \mathcal D (A) $ and $\partial \psi (\ue) = \{A \ue\}$.
\\
Let $ \ue \in\mathcal D (\partial \psi) \subseteq \mathrm{dom} (\psi) \subseteq H^1_0 (\Omega)^d \cap L^2_\sigma (\Omega)$ be given. 
Then by definition of the subgradient there exists $\bold w \in \partial \psi (\ue) \subseteq L^2_\sigma (\Omega)$ such that
\begin{align}\label{strong_solution_subdifferential_inequality_phi}
\psi (\ue) - \psi (\ve) \leq \weight{\bold w , \ue - \ve}_{L^2 (\Omega)}
\end{align}
for every $\ve \in L^2_\sigma (\Omega)$.
Now we choose $\ve := \ue + t \tilde \we$ for some $\tilde \we \in H^1_0 (\Omega)^d \cap L^2_\sigma (\Omega)$ and $ t > 0$. Then inequality (\ref{strong_solution_subdifferential_inequality_phi}) yields
\begin{align*}
\psi (\ue) - \psi (\ve) &= \into \eta (\varphi_0) D\ue : D\ue \dx - \into \eta (\varphi_0) D(\ue + t \tilde \we) : D ( \ue + t \tilde \we ) \dx \\
&= - 2 t \into \eta (\varphi_0) D\ue : D \tilde \we \dx - t^2 \into \eta (\varphi_0) D\ue : D \tilde \we \dx \\
&  \leq -t \into \we \cdot \tilde \we \dx .
\end{align*}
Dividing this inequality by $ -t < 0$ and passing to the limit $t \searrow 0$ yields
\begin{align*}
\into \we \cdot \tilde \we \dx \leq \into 2 \eta (\varphi_0) D\ue : D \tilde \we \dx  .
\end{align*}
When we replace $\tilde \we $ by $- \tilde \we$ we can conclude
\begin{align*}
\into \we \cdot \tilde \we \dx \geq \into 2 \eta (\varphi_0) D\ue : D \tilde \we \dx  .
\end{align*}
Thus it follows
\begin{align}\label{strong_solution_left_hand_side_of_lemma_subdifferential}
 \into \we \cdot \tilde \we \dx = \into 2 \eta (\varphi_0) D\ue : D \tilde \we \dx 
\end{align}
for every $\tilde \we \in H^1_0 (\Omega)^d \cap L^2_\sigma (\Omega)$. Since we assumed $\we \in L^2_\sigma (\Omega)$, we can apply \linebreak Lemma \ref{strong_solution_abels_rational_mech_lemma_4} below which yields $\ue \in H^2 (\Omega)^d \cap H^1_0 (\Omega)^d \cap L^2_\sigma (\Omega)$. Using this regularity in (\ref{strong_solution_left_hand_side_of_lemma_subdifferential}) we can conclude
\begin{align*}
 \into \we \cdot \tilde \we \dx = \into 2 \eta (\varphi_0) D\ue : D \tilde \we \dx = - \into \mathbb P_\sigma \di (2 \eta (\varphi_0) D\ue ) \cdot \tilde \we \dx
\end{align*}
for every $\tilde \we \in H^1_0 (\Omega)^d \cap L^2_\sigma (\Omega)$. Therefore, we obtain $\we = - \mathbb P _\sigma \di (2 \eta (\varphi_0) D\ue) = A\ue$ in $L^2 (\Omega)$, i.e., $\ue \in \mathcal D (A)$ and $\partial \psi (\ue) = \{ A\ue\}$. 
\end{proof}

For the regularity of the Stokes system with variable viscosity we used the following lemma.

\begin{lemma}\label{strong_solution_abels_rational_mech_lemma_4}
Let $\eta \in C^2 (\mathbb R)$ be such that $\eta (s) \geq s_0 > 0$ for all $s \in \mathbb R$ and some $s_0 > 0$, $\varphi_0 \in W^{1}_r (\Omega)$, $r > d \geq 2$, with $\|\varphi_0\|_{W^{1}_r (\Omega)} \leq R$, and let $\bold u \in H^1_0 (\Omega)^d \cap L^2_\sigma (\Omega)$ be a solution of
\begin{align*}
\weight{2 \eta (\varphi_0) D \bold u, D \boldsymbol{\tilde \we}}_{L^2 (\Omega)} = \weight{\bold w , \boldsymbol{\tilde \we}}_{L^2 (\Omega)} \qquad \text{ for all }\boldsymbol{\tilde \we} \in C^\infty_{0, \sigma} (\Omega) ,
\end{align*}
where $ \bold w \in L^2 (\Omega)^d$. Then it holds $\bold u \in H^2 (\Omega)^d$ and
\begin{align*}
\|\bold u\|_{H^2 (\Omega)} \leq C(R) \|\bold w\|_{L^2 (\Omega)} ,
\end{align*}
where $C(R)$ only depends on $\Omega$, $\eta$, $r > d$, and $R > 0$.
\end{lemma}

The proof can be found in \cite[Lemma 4]{ModelH}.

Lemma \ref{strong_solution_existence_proof_L_1_part1} implies $\ve \in W_2^1 (0,T; L^2 _\sigma (\Omega)) \cap L^\infty (0,T; H^1_0 (\Omega)^d)$. But as we want to show that $\ve$ is in $X^1_T$, it remains to show $\ve \in L^2 (0,T; H^2 (\Omega)^d)$. To this end, we also use Lemma \ref{strong_solution_abels_rational_mech_lemma_4} above. 

\begin{lemma}\label{strong_solution_h2_regularity_v}
For the unique solution $\bold v \in W_2^1 (0,T; L^2_\sigma (\Omega)) \cap L^\infty (0,T; H^1_0 (\Omega)^d)$ of  (\ref{strong_solution_L_first_equation})-(\ref{strong_solution_L_first_equation_initial_data}) from Lemma \ref{strong_solution_existence_proof_L_1_part1} it holds
$
\bold v \in L^2 (0,T; H^2 (\Omega)^d) .
$
\end{lemma}
\begin{proof}
Let $\ve \in  W_2^1 (0,T; L^2_\sigma (\Omega)) \cap L^\infty (0,T; H^1_0 (\Omega)^d)$ be the unique solution of  (\ref{strong_solution_L_first_equation}) from Lemma \ref{strong_solution_existence_proof_L_1_part1}, i.e.,
\begin{align*}
A ( \ve(t)) &= \bold f (t) - \frac{d}{dt} ( \mathcal B \ve(t)) = \bold f(t) -   \mathbb P_\sigma (\rho_0 \partial_t \ve (t))  \quad  \text{for all } 0 < t < T .
\end{align*}
Since the right-hand side is not the empty set, we get by definition of $A$
\begin{align*}
\mathbb P_\sigma (\di (2 \eta (\varphi_0) D\ve (t))) = \mathbb P_\sigma (\rho_0 \partial_t \ve (t)) - \bold f (t)    \quad \text{for all }  0 < t < T
\end{align*}
for given $\bold f \in L^2 (0,T; L^2_\sigma (\Omega))$. From $\partial_t \ve \in L^2 (0,T; L^2_\sigma (\Omega))$ it follows
\begin{align*}
\weight{ 2 \eta (\varphi_0) D\ve(t) , D \we}_{L^2 (\Omega)} = \weight{\rho_0 \partial_t \ve (t) - \bold f (t) , \we} \qquad \text{ for every } \we \in C^\infty_{0, \sigma } (\Omega) 
\end{align*}
and a.e. $t \in (0,T)$. Hence, we can apply Lemma \ref{strong_solution_abels_rational_mech_lemma_4} and obtain
\begin{align*}
  \|\ve (t)\|_{H^2 (\Omega)} &\leq C(R) \|\rho_0 \partial_t \ve (t) -  \bold f (t) \|_{L^2 (\Omega)}  \\
  &\leq C(R) \left ( \|\rho_0 \partial_t \ve (t)\|_{L^2 (\Omega)} + \|\bold f (t) \|_{L^2 (\Omega)} \right ) 
\end{align*}
for a.e. $t \in (0,T)$. Since the right-hand side of this inequality is bounded in $L^2 (0,T)$, this shows the lemma.
\end{proof}

We still need to ensure that $\|\mathcal L^{-1}\|_{\mathcal L (Y_T, X_T)} $ remains bounded. This is shown in the next lemma.

\begin{lemma}\label{strong_solution_lemma_l_inverse_t_t_0_firstpart}
Let the assumptions of Lemma \ref{strong_solution_existence_proof_L_1_part1} hold and $0 < T_0 < \infty$ be given. Then 
\begin{align*}
\|\mathcal L^{-1}_{1, T}\|_{\mathcal L (Y^1_T, X^1_T)} \leq \|\mathcal L^{-1}_{1, T_0}\|_{\mathcal L (Y^1_{T_0}, X^1_{T_0})} < \infty \qquad \text{ for all } 0 < T < T_0 .
\end{align*}
\end{lemma}
\begin{proof}
Let $ 0 < T < T_0 $ be given. Lemma \ref{strong_solution_existence_proof_L_1_part1} together with Lemma \ref{strong_solution_h2_regularity_v} yields that the operator $\mathcal L_{1, T} : X_T \rightarrow Y_T$ is invertible for every $0 < T < \infty$ and every given  \linebreak $\bold f \in L^2(0,T; L^2_\sigma (\Omega))$, $\varphi_0 \in W^1_r (\Omega)$, $\ve_0 \in H^1_0 (\Omega)^d \cap L^2_\sigma (\Omega)$. Then we define \linebreak $\bold{ \tilde f } \in L^2 (0,T_0, L^2_\sigma (\Omega))$ by
\begin{align*}
\bold{\tilde f} (t) := 
\begin{cases}
\bold f(t)  & \text{ if } t \in (0,T] , \\
0  & \text{ if } t \in (T, T_0 ) .
\end{cases}
\end{align*}
Due to Lemma \ref{strong_solution_existence_proof_L_1_part1} together with Lemma \ref{strong_solution_h2_regularity_v} there exists a unique solution $\tilde \ve \in X^1_{T_0}$ of
\begin{align*}
\mathbb P_\sigma (\rho_0 \partial_t  \tilde \ve) - \mathbb P_\sigma (\di (2 \eta (\varphi_0) D \tilde \ve)) &= \bold{\tilde f}   && \text{ in } Q_{T_0},   \\
\di ( \tilde \ve) &= 0 && \text{ in } Q_{T_0} , \\
\tilde \ve_{|\partial \Omega} &= 0 && \text{ on } (0, T_0) \times  \partial \Omega  , \\
\tilde \ve (0) &= \ve_0  && \text{ in } \Omega .
\end{align*}
So let $\ve \in X^1_T$ be the solution of the previous equations with $T_0$ replaced by $T$. Then $\tilde \ve$ and $\ve$ solve these equations on $(0,T) \times \Omega$. Since the solution is unique, we can deduce $\tilde \ve_{|(0,T)} = \ve$. Hence,
\begin{align*}
\|\mathcal L^{-1}_{1, T} (\bold f) \|_{X^1_T} &= \| \ve\|_{X^1_{T}} \leq \|\tilde \ve\|_{X^1_{T_0}} = \|\mathcal L^{-1}_{1, T_0} (\bold{\tilde f})\|_{X^1_{T_0}} \\
& \leq \|\mathcal L^{-1}_{1, T_0}\|_{\mathcal L (Y^1_{T_0}, X^1_{T_0})} \|\bold{\tilde f}\|_{Y^1_{T_0}} = \|\mathcal L^{-1}_{1, T_0}\|_{\mathcal L (Y^1_{T_0}, X^1_{T_0})} \| \bold f\|_{Y^1_{T}} ,
\end{align*}
which shows the statement since it holds $   \|\mathcal L^{-1}_{1, T_0}\|_{\mathcal L (Y^1_{T_0}, X^1_{T_0})} < \infty$ by the bounded inverse theorem.
\end{proof}

Finally, we have to show invertibility of the second part of $\mathcal{L}$.

\begin{lemma}\label{strong_solution_existence_proof_L_1_part2}
Let Assumption \ref{strong_solutions_general_assumptions} hold and $\varphi_0 \in (L^p (\Omega) , W^4_{p,N} (\Omega))_{1 - \frac{1}{p}, p}$, $f \in L^p (0,T; L^p (\Omega)) $ with $4 < p < 6$ be given. Then for every $0 < T < \infty$ there exists a unique 
\begin{align*} 
\varphi \in  L^p (0,T; W^4_{p,N} (\Omega)) \cap \{ u \in W^1_p(0,T; L^p (\Omega))  : \ u_{|t=0} = \varphi_0 \}
\end{align*} 
such that 
\begin{align}
\partial_t \varphi + \varepsilon m (\varphi_0) \Delta^2 \varphi &= f  && \text{ in } (0,T) \times \Omega , \label{strong_solution_second_equation_initial_equation} \\
\partial_n \varphi_{|\partial \Omega} &= 0     && \text{ on } (0,T) \times \partial \Omega  ,\\
\partial_n \Delta \varphi_{|\partial \Omega} & = 0  && \text{ on } (0,T) \times \partial \Omega ,  \\
\varphi (0) &= \varphi_0 && \text{ in } \{ 0 \} \times \Omega . \label{strong_solution_second_equation_initial_condition_start}
\end{align}
\end{lemma}

\begin{proof}
  The result follows from standard results on maximal regularity of parabolic equations, e.g.\ from \cite[Theorem 8.2]{MR2006641}.
\end{proof}

Analogously to the previous part we need to ensure that $\|\mathcal L^{-1}\|_{\mathcal L (Y_T, X_T)} $ remains bounded.

\begin{lemma}\label{strong_solution_lemma_l_inverse_t_t_0_secondpart}
Let the assumptions of Lemma \ref{strong_solution_existence_proof_L_1_part2} hold and $0 < T_0 < \infty$ be given. Then 
\begin{align*}
\|\mathcal L^{-1}_{2, T}\|_{\mathcal L (Y^2_T, X^2_T)} \leq \|\mathcal L^{-1}_{2, T_0}\|_{\mathcal L (Y^2_{T_0}, X^2_{T_0})} < \infty \qquad \text{ for all } 0 < T < T_0 .
\end{align*}
\end{lemma}
This lemma can be proven analogously to Lemma \ref{strong_solution_lemma_l_inverse_t_t_0_firstpart}.

From the results of this section Theorem~\ref{thm:linear} follows immediately.



\def\cprime{$'$} \def\ocirc#1{\ifmmode\setbox0=\hbox{$#1$}\dimen0=\ht0
  \advance\dimen0 by1pt\rlap{\hbox to\wd0{\hss\raise\dimen0
  \hbox{\hskip.2em$\scriptscriptstyle\circ$}\hss}}#1\else {\accent"17 #1}\fi}
\providecommand{\bysame}{\leavevmode\hbox to3em{\hrulefill}\thinspace}
\providecommand{\MR}{\relax\ifhmode\unskip\space\fi MR }
\providecommand{\MRhref}[2]{%
  \href{http://www.ams.org/mathscinet-getitem?mr=#1}{#2}
}
\providecommand{\href}[2]{#2}


\end{document}